\documentclass{tsp-short}

\usepackage[utf8]{inputenc}
\usepackage{amsmath,amsthm,enumitem,mathtools,mathrsfs,booktabs}
\usepackage{tikz}

\theoremstyle{plain}
\newtheorem{theorem}{Theorem}[section]
\newtheorem{lemma}[theorem]{Lemma}
\newtheorem{proposition}[theorem]{Proposition}
\newtheorem{corollary}[theorem]{Corollary}
\theoremstyle{definition}
\newtheorem{definition}[theorem]{Definition}
\newtheorem{assumption}[theorem]{Assumption}
\theoremstyle{remark}
\newtheorem{remark}[theorem]{Remark}

\usepackage{xcolor}
\usepackage[colorlinks=true, pdfencoding=auto, psdextra]{hyperref}
\hypersetup{urlcolor=blue, citecolor=green, linkcolor=red}

\newcommand{\R}{\mathbb{R}}
\newcommand{\N}{\mathbb{N}}
\newcommand{\E}{\mathbb{E}}
\DeclareMathOperator{\sign}{sign}
\newcommand{\al}{\alpha}
\newcommand{\be}{\beta}
\newcommand{\ga}{\gamma}
\newcommand{\eps}{\varepsilon}
\newcommand{\Ga}{\Gamma}
\newcommand{\De}{\Delta}

\setlist{nosep}
\numberwithin{equation}{section}

\begin{document}

\title
[Higher-order multifractional stable motion]
{Higher-order multifractional stable motion:\\ definition and dependence structure}

\author{Atef Lechiheb}
\address{D\'epartement de math\'ematiques, Toulouse School of Economics, Universit\'e Toulouse Capitole, Toulouse, France}
\email{atef.lechiheb@tse-fr.eu}
\urladdr{https://orcid.org/0000-0002-9896-3408}



\begin{abstract}
We introduce the $n$-th order multifractional stable motion ($n$-MFSM), a
symmetric $\al$-stable process ($\al\in(1,2]$) whose Hurst parameter is a
function $H(t)\in(n-1,n)$, obtained by a Taylor regularization of order
$n\ge1$ of the kernel of the linear multifractional stable motion. It
contains the $n$-th order fractional stable motion of Kawai as the special
case of constant $H$, and, up to explicit deterministic factors, the linear
multifractional stable motion ($n=1$) and the moving-average $n$-th order
multifractional Brownian motion ($\al=2$). We prove existence, give the
harmonizable representation in the Gaussian case, and show that, if $H$ is
of class $\mathcal C^n$ near $t_0$, the rescaled $n$-th order increments at
$t_0$ converge in finite-dimensional distributions to the $n$-th order
increments of the $n$-th order fractional stable motion of index $H(t_0)$.
We then study the dependence structure through the codifference. For $s>0$
fixed and $H(t)\to H_+$ with $(H(t)-H_+)\log t\to0$, we obtain the exact
power rate of $\tau(X(s),X(t))-c_0(s)$ as $t\to\infty$ away from the
boundaries of three explicit parameter regimes, where
$c_0(s)=\|X(s)\|_\al^\al$ if $\al<2$ and $c_0=0$ if $\al=2$; the constants
are explicit and of known sign when $n\ge2$ or $\al=2$, while for $n=1$ and
$\al<2$ part of one regime yields only convergence or a modified constant.
In particular this codifference does not tend to zero when $n\ge2$. For the
$n$-th order increments with constant $H$, whose codifference asymptotics
were obtained by Kawai, we give a self-contained proof with the sign of the
leading constants and treat the boundary cases of the long-range dependence
criterion $H\ge n-1/\al$; for a non-constant Hurst function this criterion
is extended only to functions that are eventually constant.
\end{abstract}

\maketitle

\section{Introduction}
\label{sec:intro}

Many observed signals combine three features: heavy-tailed marginals,
a local regularity that changes in time, and a degree of smoothness that
can exceed that of a fractional Brownian path. Heavy tails are modelled by
$\al$-stable distributions \cite{SamorTaqqu1994}; heavy tails and scaling
are both present, for instance, in network traffic \cite{WillingerTaqqu1997}.
Time-varying regularity is the motivation behind multifractional processes,
used for instance for financial prices \cite{Frezza2022} and geophysical
logs \cite{Gaci2011}; see also \cite{Lopes2009} for fractal methods in
imaging. Hurst parameters larger than one arise for integrated or
cumulated signals, which is the motivation of the $n$-th order models of
\cite{PerrinHarba2001}. Heavy tails together with non-constant local scaling
have also been reported for the motion of molecules on cell membranes
\cite{ShengChen2011}.

\subsection*{Existing models}

Fractional Brownian motion \cite{MandelbrotVanNess1968} is, up to a
multiplicative constant, the unique centred Gaussian $H$-self-similar process
with stationary increments, $H\in(0,1)$.
Multifractional Brownian motion \cite{PeltierLevyVehel1995,Benassi1997}
replaces $H$ by a function $H(t)$; see \cite{StoevTaqqu2006} for a broad
class of such processes. In the stable setting, the (causal) linear
fractional stable motion (LFSM)
\[
L_{H,\al}(t)=\int_\R\bigl[(t-u)_+^{H-1/\al}-(-u)_+^{H-1/\al}\bigr]dM_\al(u)
\]
is the moving-average analogue of fractional Brownian motion, and the linear
multifractional stable motion (LMSM) of Stoev and Taqqu \cite{StoevTaqqu2004}
lets $H$ depend on $t$. Its path properties were studied in
\cite{AyacheHamonier2014,AyacheHamonier2017}, the estimation of $H(\cdot)$
and $\al$ in \cite{AyacheHamonier2015,Dang2020}, local times in
\cite{DingPengXiao2023}, and a harmonizable counterpart in
\cite{DozziShevchenko2011}. Multistable and tempered variants are considered
in \cite{FalconerLevyVehel2009,FanLevyVehel2019}. Recently, Mies and Sikkens
\cite{MiesSikkens2026} introduced a first-order multifractional stable motion
in which the Hurst function enters through the integration variable and may
be random, and identified its tangent process and its local Hölder exponent.

Hurst parameters beyond $(0,1)$ are obtained by subtracting from the kernel
the first terms of its Taylor expansion. This gives the $n$-th order
fractional Brownian motion of Perrin et al.\ \cite{PerrinHarba2001},
$H\in(n-1,n)$, its multifractional version \cite{GuptaKumarLeonenko2022},
and, in the stable case with constant $H$, the higher order fractional stable
motion of Kawai \cite{Kawai2016}. Kawai proves existence, self-similarity,
path properties and the stationarity of the $n$-th order increments, and he
obtains the asymptotics of the codifference of these increments, together
with the corresponding long-range dependence criterion
\cite[Theorem~5.2 and Corollary~5.4]{Kawai2016}.

\subsection*{The $n$-MFSM and our results}

To our knowledge, none of these models combines heavy tails, a functional
Hurst parameter and higher order. For $\al\in(1,2]$, $n\in\N^*$ and $H$ as in
Assumption~\ref{ass:params} we define
\[
X^{(n)}_H(t)=\int_\R f_n(t,u)\,dM_\al(u),
\]
where $f_n$ is the Taylor-regularized kernel \eqref{eq:kernel}. Our main
results are the following.

\begin{enumerate}[label=\textbf{(\arabic*)},leftmargin=2.2em]
\item \textbf{Existence and representations} (Section~\ref{sec:definition}).
$X^{(n)}_H$ is well defined, its $n$-th order increments of any step are
stationary when $H$ is constant, and in the Gaussian case $\al=2$ we give its
harmonizable representation with an explicit kernel
(Theorem~\ref{thm:harmonizable}). For $\al<2$ the harmonizable analogue is a
different process (Remark~\ref{rem:harm-stable}).

\item \textbf{Local structure} (Section~\ref{sec:local}).
If $H$ is of class $\mathcal C^n$ on a neighbourhood of $t_0$, the rescaled $n$-th order
increments $\eps^{-H(t_0)}\De^n_{\eps\tau}X^{(n)}_H(t_0)$ converge in
finite-dimensional distributions to the $n$-th order increments of the
$n$-th order fractional stable motion with index $H(t_0)$
(Theorem~\ref{thm:LASS}).

\item \textbf{Codifference of the process} (Section~\ref{sec:LRD}).
Let $s>0$, $H_+=\lim_{t\to\infty}H(t)$ with $(H(t)-H_+)\log t\to0$, and
$c_0(s)=\|X^{(n)}_H(s)\|_\al^\al$ if $\al<2$, $c_0=0$ if $\al=2$. Away from
the boundaries of three explicit regimes, which are shown to cover all other
parameter values (Lemma~\ref{lem:regimes}),
$\tau(X^{(n)}_H(s),X^{(n)}_H(t))-c_0(s)\sim C(s)\,t^{\rho(s)}$ as $t\to\infty$,
with an exponent $\rho(s)$ that takes three different forms according to
$\al$, $n$, $H(s)$ and $H_+$, and explicit constants, of known sign when
$n\ge2$ or $\al=2$ (Theorem~\ref{thm:codiff-asymp}). For $n=1$ and $\al<2$
the theorem gives the exact rate in regime E1, the limit in regime E2, and in
regime E4 the limit together with the rate when $H(s)>\frac1\al-\frac12$; the
boundaries of the regimes are not covered (Remark~\ref{rem:boundary}). In particular the codifference of the
process does not tend to zero when $n\ge2$ (Corollary~\ref{cor:no-decay}).

\item \textbf{Memory of the increments} (Section~\ref{sec:increments}).
For constant $H\equiv H_0$ and $\al\in(1,2)$, the asymptotics of the
codifference $r(m)$ of the $n$-th order increments, with the threshold
$H_*=n-\frac{1}{\al(\al-1)}$ and the exponents $\al(H_0-n)$ and
$H_0-n-1/\al$, are due to Kawai \cite[Theorem~5.2]{Kawai2016}, extending the
case $n=1$ of Astrauskas, Levy and Taqqu \cite{AstrauskasLevyTaqqu1991}. We
give a self-contained proof (Theorem~\ref{thm:incr}), show that the leading
constant is positive for $H_0>H_*$ and, for $n=1$, for $H_0<H_*$ (the sign
for $n\ge2$ is supported numerically), and complement the long-range
dependence criterion of \cite[Corollary~5.4]{Kawai2016} by the boundary values
$H_0=n-1/\al$ and $H_0=H_*$ (Corollary~\ref{cor:LRD-incr}). For a
non-constant Hurst function we only treat the case where $H$ is constant
after some time $T$: the criterion then carries over to the increments after
time $T$ (Proposition~\ref{prop:eventually-constant}). For a Hurst function
that merely converges at infinity we have no result
(Remark~\ref{rem:open-incr}).
\end{enumerate}

The main new results are therefore the definition of $n$-MFSM itself, the
local asymptotic self-similarity of its increments for $n\ge2$
(Theorem~\ref{thm:LASS}), and the codifference of the process itself
(Theorem~\ref{thm:codiff-asymp}), which has no counterpart in
\cite{Kawai2016}.

\subsection*{Notation}
For $x\in\R$, $x_+=\max\{x,0\}$, and for $\be\in\R$ we set
$x_+^\be:=x^\be\mathbf 1_{\{x>0\}}$ (so $0_+^\be=0$); we also use the
convention $t^0:=1$ for all $t\in\R$. For $\ga\in\R$ and $k\in\N$,
$\binom{\ga}{k}=\ga(\ga-1)\cdots(\ga-k+1)/k!$, with $\binom{\ga}{0}=1$.
We write $f\sim g$ when $f/g\to1$, and $\overset{\mathrm{fdd}}{\longrightarrow}$
for convergence of finite-dimensional distributions. For $h\in\R$, the
$n$-th order increment of a function or process $Z$ is
\begin{equation}
\label{eq:Delta-n}
\De^n_hZ(t):=\sum_{j=0}^n c_j\,Z(t+jh),\qquad c_j:=(-1)^{n-j}\binom nj .
\end{equation}
Stochastic integrals with respect to a S$\al$S random measure $M_\al$ are
understood as in \cite[Chapter~3]{SamorTaqqu1994}. All processes are
considered through their finite-dimensional distributions; no path property
is asserted, and equalities between processes are equalities of
finite-dimensional distributions unless stated otherwise. We use three
different letters: $X^{(n)}_H$ for the $n$-MFSM with Hurst function $H$,
$Y^{(n)}_{H_0}$ for the $n$-th order fractional stable motion with constant
index $H_0$, and $(Z_k)_{k\in\mathbb Z}$ for the sequence of $n$-th order
increments with unit step studied in Section~\ref{sec:increments}.

\section{Preliminaries}
\label{sec:prelim}

We fix the Fourier transform convention
$\widehat f(\xi)=\int_\R f(x)e^{-ix\xi}\,dx$.

\subsection{Symmetric $\al$-stable random variables and integrals}

A real random variable $\xi$ is S$\al$S with scale $\sigma\ge0$ if
$\E[e^{i\theta\xi}]=e^{-\sigma^\al|\theta|^\al}$, $\theta\in\R$; we write
$\|\xi\|_\al:=\sigma$. For $\al=2$ this means $\xi\sim\mathcal N(0,2\sigma^2)$,
so $\mathrm{Var}(\xi)=2\|\xi\|_2^2$. For $\al\in(1,2)$, $\sigma>0$ and $p>0$,
$\E|\xi|^p<\infty$ if and only if $p<\al$ \cite[Chapter~1]{SamorTaqqu1994}.

Let $M_\al$ be a S$\al$S random measure on $\R$ with Lebesgue control measure:
for Borel sets $A$ of finite measure,
$\E[e^{i\theta M_\al(A)}]=e^{-|\theta|^\al|A|}$, the values on disjoint sets are
independent, and for disjoint $A_1,A_2,\ldots$ with $|\bigcup_jA_j|<\infty$,
$M_\al(\bigcup_jA_j)=\sum_jM_\al(A_j)$ almost surely. Such a measure
exists and its law is unique \cite[Chapter~3]{SamorTaqqu1994}.

\begin{theorem}[\cite{SamorTaqqu1994}, Chapter~3]
\label{thm:SaS-integral}
Let $\al\in(1,2]$.
\begin{enumerate}[label=(\roman*)]
\item For $f\in L^\al(\R)$, $I(f)=\int_\R f\,dM_\al$ is a S$\al$S random
variable with $\|I(f)\|_\al=\|f\|_{L^\al}$.
\item $I$ is linear.
\item For $f_1,\ldots,f_m\in L^\al(\R)$ and $\theta\in\R^m$,
\begin{equation}
\label{eq:fdd-integral}
\E\Bigl[\exp\Bigl(i\sum_{j=1}^m\theta_jI(f_j)\Bigr)\Bigr]
=\exp\Bigl(-\int_\R\Bigl|\sum_{j=1}^m\theta_jf_j(u)\Bigr|^\al du\Bigr).
\end{equation}
\item If $\|f_k-f\|_{L^\al}\to0$, then $I(f_k)\to I(f)$ in probability.
\end{enumerate}
\end{theorem}

Two consequences of \eqref{eq:fdd-integral} will be used repeatedly. If
$g_1,\ldots,g_m\in L^\al(\R)$, $c>0$ and $t_0\in\R$, then
\begin{align}
\bigl(I(g_j(\cdot-t_0))\bigr)_{j\le m}&\overset{d}{=}\bigl(I(g_j)\bigr)_{j\le m},
\label{eq:translation}\\
\bigl(I(g_j(c\,\cdot))\bigr)_{j\le m}&\overset{d}{=}\bigl(c^{-1/\al}I(g_j)\bigr)_{j\le m},
\label{eq:dilation}
\end{align}
since the right-hand side of \eqref{eq:fdd-integral} is invariant under
translations, and is multiplied by $c^{-1}$ inside the exponential under
the dilation $u\mapsto cu$.

\subsection{Fourier transform of power functions}

For $z\in\R$ and $\xi\neq0$ we write
$(i\xi)^{z}:=|\xi|^{z}e^{i\pi z\sign(\xi)/2}$ and
$(-i\xi)^{z}:=|\xi|^{z}e^{-i\pi z\sign(\xi)/2}$, so that
$(\pm i\xi)^{z_1}(\pm i\xi)^{z_2}=(\pm i\xi)^{z_1+z_2}$; for
$z\in\mathbb Z$ these coincide with the usual powers of $\pm i\xi$.

\begin{lemma}[Fourier transform of $x_+^\be$]
\label{lem:FT-power}
Let $\be>-1$. The tempered distribution $x\mapsto x_+^\be$ has a Fourier
transform which coincides on $\R\setminus\{0\}$ with the function
$\xi\mapsto\Ga(\be+1)(i\xi)^{-\be-1}$. If $-1<\be<0$, then
$\int_0^\infty x^\be e^{-ix\xi}dx=\Ga(\be+1)(i\xi)^{-\be-1}$ as an improper
integral, for every $\xi\neq0$.
\end{lemma}

\begin{proof}
For $z\in\mathbb C$ with $\Re z>0$, put $\Lambda(z):=\int_0^\infty x^\be e^{-zx}dx$.
Both $\Lambda(z)$ and $\Ga(\be+1)z^{-\be-1}$ (principal branch) are analytic
on $\{\Re z>0\}$ and coincide for real $z>0$; hence
$\Lambda(\eta+i\xi)=\Ga(\be+1)(\eta+i\xi)^{-\be-1}$ for $\eta>0$, $\xi\in\R$.
As $\eta\downarrow0$, $x_+^\be e^{-\eta x}\to x_+^\be$ in $\mathcal S'(\R)$, so
their Fourier transforms converge in $\mathcal S'(\R)$ by continuity of the
Fourier transform on $\mathcal S'$; on the other hand
$(\eta+i\xi)^{-\be-1}\to(i\xi)^{-\be-1}$ locally uniformly on
$\R\setminus\{0\}$. Hence the Fourier transform of $x_+^\be$ coincides on
$\R\setminus\{0\}$ with $\Ga(\be+1)(i\xi)^{-\be-1}$
\cite[Chapter~II]{GelfandShilov1964}. When $-1<\be<0$ and $\xi\neq0$, the
improper integral $\int_0^\infty x^\be e^{-ix\xi}dx$ converges by Dirichlet's
test, and by Abel's theorem for Laplace transforms it equals
$\lim_{\eta\downarrow0}\Lambda(\eta+i\xi)=\Ga(\be+1)(i\xi)^{-\be-1}$.
\end{proof}

\subsection{The codifference}
\label{subsec:codiff-prelim}

For jointly S$\al$S random variables $X,Y$ the codifference is
\begin{equation}
\label{eq:codiff}
\tau(X,Y):=\|X\|_\al^\al+\|Y\|_\al^\al-\|X-Y\|_\al^\al ,
\end{equation}
see e.g.\ \cite{SamorTaqqu1994,PipirasTaqqu2017}.

\begin{lemma}[Properties of the codifference]
\label{lem:codiff-finite}
Let $\al\in(1,2]$, $X=I(f_1)$, $Y=I(f_2)$ with $f_1,f_2\in L^\al(\R)$. Then:
\begin{enumerate}[label=(\roman*)]
\item $\tau(X,Y)=\int_\R\bigl[|f_1|^\al+|f_2|^\al-|f_1-f_2|^\al\bigr]du$;
\item $|\tau(X,Y)|\le(1+\al)\|X\|_\al^\al+\al\|X\|_\al\|Y\|_\al^{\al-1}$;
\item $\tau(X,Y)=\tau(Y,X)$, and $\tau(X,Y)=0$ if $X$ and $Y$ are independent;
\item if $\al=2$, $\tau(X,Y)=\mathrm{Cov}(X,Y)$.
\end{enumerate}
The codifference may be negative: for $Y=-X\neq0$ and $\al>1$,
$\tau(X,-X)=(2-2^\al)\|X\|_\al^\al<0$.
\end{lemma}

\begin{proof}
(i) By linearity $X-Y=I(f_1-f_2)$, and Theorem~\ref{thm:SaS-integral}(i)
gives the three scales. (ii) For $a,b\in\R$, by the mean value theorem
$\bigl||b|^\al-|b-a|^\al\bigr|\le\al|a|\,(|a|+|b|)^{\al-1}\le\al|a|\bigl(|a|^{\al-1}+|b|^{\al-1}\bigr)$,
hence $\bigl||a|^\al+|b|^\al-|a-b|^\al\bigr|\le(1+\al)|a|^\al+\al|a||b|^{\al-1}$.
We integrate with $a=f_1$, $b=f_2$ and use Hölder's inequality
$\int|f_1||f_2|^{\al-1}\le\|f_1\|_{L^\al}\|f_2\|_{L^\al}^{\al-1}$.
(iii) Symmetry is clear. If $X,Y$ are independent,
$\E e^{i\theta(X-Y)}=e^{-|\theta|^\al(\|X\|^\al_\al+\|Y\|^\al_\al)}$, so
$\|X-Y\|_\al^\al=\|X\|_\al^\al+\|Y\|_\al^\al$.
(iv) With $\mathrm{Var}=2\|\cdot\|_2^2$,
$\tau(X,Y)=\frac12[\mathrm{Var}X+\mathrm{Var}Y-\mathrm{Var}(X-Y)]=\mathrm{Cov}(X,Y)$.
\end{proof}

For a stationary S$\al$S sequence, the non-summability of the codifference
is one of the usual ways of expressing long-range dependence; see
\cite{PipirasTaqqu2017} for a discussion of the various notions. Since $n$-MFSM is not stationary, we study in
Section~\ref{sec:LRD} the codifference of the process itself and in
Section~\ref{sec:increments} that of its $n$-th order increments.

\section{The $n$-th order multifractional stable motion}
\label{sec:definition}

\subsection{Parameters and kernel}

\begin{assumption}
\label{ass:params}
$\al\in(1,2]$, $n\in\N^*$, and $H:\R\to\R$ is measurable with
$n-1<a_H\le H(t)\le b_H<n$ for all $t$, for some constants $a_H,b_H$.
\end{assumption}

Additional regularity of $H$ is required only in Section~\ref{sec:local}.
We set $\ga(t):=H(t)-1/\al$, so that
\begin{equation}
\label{eq:gamma-range}
\ga(t)\in\bigl(n-1-\tfrac1\al,\;n-\tfrac1\al\bigr),\qquad\ga(t)>-\tfrac1\al>-1 .
\end{equation}

\begin{definition}[Regularized kernel]
\label{def:kernel}
For $t,u\in\R$,
\begin{equation}
\label{eq:kernel}
f_n(t,u):=\frac{1}{\Ga(\ga(t)+1)}\Bigl[(t-u)_+^{\ga(t)}
-\sum_{k=0}^{n-1}\binom{\ga(t)}{k}t^k(-u)_+^{\ga(t)-k}\Bigr].
\end{equation}
We write $f_n(t,u;\ga)$ for the same expression with $\ga(t)$ replaced by a
number $\ga$.
\end{definition}

For $u=-v<0$ and fixed $t$, the binomial series gives
$(t+v)^{\ga}=\sum_{k\ge0}\binom{\ga}{k}t^kv^{\ga-k}$ for $v>|t|$: the subtracted
sum in \eqref{eq:kernel} removes the first $n$ terms, which yields the decay
$v^{\ga-n}$, and hence the $L^\al$-integrability at $-\infty$, when $H<n$. The
coefficient $\binom{\ga}{n}$ of the leading term has the sign of
$\ga-n+1$, since the other factors $\ga-j$, $j\le n-2$, are positive. When
$\ga=n-1$ (that is, $H=n-1+1/\al$) the binomial series terminates; the
kernel is then not zero, but it has compact support, see
Remark~\ref{rem:integer-gamma}.

\begin{lemma}[Support and tail of the kernel]
\label{lem:kernel-asymp}
Let $t\in\R$ and $\ga=\ga(t)$.
\begin{enumerate}[label=(\roman*)]
\item $f_n(t,u)=0$ for $u>\max\{0,t\}$.
\item For $v>0$ with $v\ge2|t|$,
\begin{equation}
\label{eq:kernel-tail}
f_n(t,-v)=\frac{1}{\Ga(\ga+1)}\sum_{k=n}^\infty\binom{\ga}{k}t^kv^{\ga-k}
=C_n(t)\,v^{\ga-n}+\mathcal R_n(t,v),
\end{equation}
with $C_n(t):=\dfrac{t^n}{\Ga(\ga+1)}\dbinom{\ga}{n}$
and $|\mathcal R_n(t,v)|\le K_\ga|t|^{n+1}v^{\ga-n-1}$, where
$K_\ga:=\Ga(\ga+1)^{-1}\sum_{k\ge0}\bigl|\binom{\ga}{n+1+k}\bigr|2^{-k}<\infty$
is bounded for $\ga$ in compact subsets of $(-1,\infty)$. The proof of this
statement only uses $\ga>-1$, so it holds for $f_n(t,\cdot\,;g)$ with any
$g>-1$.
\end{enumerate}
\end{lemma}

\begin{proof}
(i) If $u>\max\{0,t\}$, then $(t-u)_+=(-u)_+=0$. (ii) For $u=-v<0$ we have
$(-u)_+=v$ and $t-u=t+v>0$. For $v\ge2|t|$,
$(t+v)^\ga=v^\ga(1+t/v)^\ga=\sum_{k\ge0}\binom{\ga}{k}t^kv^{\ga-k}$, the series
converging absolutely since $|t/v|\le\frac12$. The subtracted terms in
\eqref{eq:kernel} are exactly the terms $k\le n-1$, which gives the series.
Then
\begin{align*}
\Bigl|\sum_{k\ge n+1}\binom{\ga}{k}t^kv^{\ga-k}\Bigr|
&\le|t|^{n+1}v^{\ga-n-1}\sum_{k\ge0}\Bigl|\binom{\ga}{n+1+k}\Bigr|\Bigl(\frac{|t|}{v}\Bigr)^{k}\\
&\le|t|^{n+1}v^{\ga-n-1}\sum_{k\ge0}\Bigl|\binom{\ga}{n+1+k}\Bigr|2^{-k},
\end{align*}
and the last series converges because
$|\binom{\ga}{k+1}|/|\binom{\ga}{k}|=|\ga-k|/(k+1)\to1$.
\end{proof}

\subsection{Definition and existence}
\label{subsec:existence}

\begin{definition}[$n$-th order multifractional stable motion]
\label{def:nMFSM}
Under Assumption~\ref{ass:params}, the $n$-MFSM is
\begin{equation}
\label{eq:nMFSM}
X^{(n)}_H(t):=\int_\R f_n(t,u)\,dM_\al(u),\qquad t\in\R .
\end{equation}
When $H\equiv H_0$ is constant we write $Y^{(n)}_{H_0}$ for $X^{(n)}_H$ and call
it the $n$-th order fractional stable motion ($n$-FSM). Its kernel and
normalization are those of \cite[Definition~3.1]{Kawai2016}, where the
driving process is a symmetric $\al$-stable Lévy process $L_\al$ with
$L_\al(1)\sim\mathrm{S}\al\mathrm{S}(1)$, so that $Y^{(n)}_{H_0}$ coincides in law
on $[0,\infty)$ with Kawai's process (Kawai excludes the values of $H_0$ for
which $H_0-1/\al$ is an integer).
\end{definition}

\begin{theorem}[Existence and basic properties]
\label{thm:existence}
Under Assumption~\ref{ass:params}:
\begin{enumerate}[label=(\roman*)]
\item $f_n(t,\cdot)\in L^\al(\R)$ for every $t$, so $X^{(n)}_H(t)$ is a
well-defined S$\al$S random variable with
$\|X^{(n)}_H(t)\|_\al^\al=\int_\R|f_n(t,u)|^\al du$;
\item the finite-dimensional distributions are given by
\eqref{eq:fdd-integral} with $f_j=f_n(t_j,\cdot)$; since all the variables
$X^{(n)}_H(t_j)$ are integrals of deterministic functions with respect to the
same independently scattered S$\al$S random measure $M_\al$, every linear
combination $\sum_j\theta_jX^{(n)}_H(t_j)=\int_\R\sum_j\theta_jf_n(t_j,u)\,dM_\al(u)$
is S$\al$S, so the vector $(X^{(n)}_H(t_j))_j$ is jointly S$\al$S;
\item if $H\equiv H_0$ is constant, then $Y^{(n)}_{H_0}(0)=0$, $Y^{(n)}_{H_0}$ is
$H_0$-self-similar, and its $n$-th order increments
$(\De^n_hY^{(n)}_{H_0}(t))_{t\in\R}$ are stationary for every $h$ (on
$[0,\infty)$ these properties are in \cite[Theorem~4.1]{Kawai2016}).
\end{enumerate}
\end{theorem}

\begin{proof}
(i) Fix $t$ and $R\ge2|t|+1$. On $(-R,R)$, by \eqref{eq:kernel},
\[
|f_n(t,u)|\le\frac{1}{\Ga(\ga+1)}\Bigl[(t-u)_+^{\ga}+\sum_{k=0}^{n-1}\Bigl|\binom{\ga}{k}\Bigr||t|^k(-u)_+^{\ga-k}\Bigr],
\]
and each term is in $L^\al(-R,R)$ because $\al\ga>-1$ and
$\al(\ga-k)\ge\al(\ga-n+1)>-1$ by \eqref{eq:gamma-range}. On $(-\infty,-R]$,
Lemma~\ref{lem:kernel-asymp}(ii) gives $|f_n(t,-v)|\le C\,v^{\ga-n}$ with
$C=|C_n(t)|+K_\ga|t|^{n+1}$ (since $v\ge1$), and
$\al(\ga-n)=\al(H(t)-n)-1<-1$. On $[R,\infty)$, $f_n(t,\cdot)=0$ by
Lemma~\ref{lem:kernel-asymp}(i). The rest follows from
Theorem~\ref{thm:SaS-integral}.

(ii) This is Theorem~\ref{thm:SaS-integral}(iii), together with the
linearity of the integral.

(iii) With $\ga_0=H_0-1/\al$, $f_n(0,u;\ga_0)=\Ga(\ga_0+1)^{-1}[(-u)_+^{\ga_0}-(-u)_+^{\ga_0}]=0$.
Each term of \eqref{eq:kernel} is homogeneous of degree $\ga_0$ in $(t,u)$,
so $f_n(ct,cu;\ga_0)=c^{\ga_0}f_n(t,u;\ga_0)$ for $c>0$. By
\eqref{eq:dilation} with $c$ replaced by $1/c$,
$(Y(ct_j))_j=(I(f_n(ct_j,\cdot)))_j=(c^{\ga_0}I(f_n(t_j,\cdot/c)))_j
\overset{d}{=}(c^{\ga_0+1/\al}Y(t_j))_j=(c^{H_0}Y(t_j))_j$.
For the increments, fix $h\in\R$. For $k\le n-1$, $j\mapsto(t+jh)^k$ is a
polynomial of degree $k<n$ in $j$, so $\sum_jc_j(t+jh)^k=0$ (each forward
difference lowers the degree of a polynomial by one, so $n$ differences
annihilate every polynomial of degree at most $n-1$); hence the
polynomial terms of \eqref{eq:kernel} cancel in $\sum_jc_jf_n(t+jh,u;\ga_0)$,
which is the kernel of $\De^n_hY(t)$, and
\[
\sum_{j=0}^nc_jf_n(t+jh,u;\ga_0)=g_h(t-u),\qquad
g_h(x):=\frac{1}{\Ga(\ga_0+1)}\sum_{j=0}^nc_j(x+jh)_+^{\ga_0}.
\]
By \eqref{eq:translation}, for $t_1,\ldots,t_m,s\in\R$ and steps
$h_1,\ldots,h_m$, $(\De^n_{h_i}Y(t_i+s))_i\overset{d}{=}(\De^n_{h_i}Y(t_i))_i$.
\end{proof}

\begin{remark}[The value $H(t)=n-1+\frac1\al$]
\label{rem:integer-gamma}
Assumption~\ref{ass:params} allows $\ga(t)=n-1$, that is,
$H(t)=n-1+\frac1\al$, and this value is deliberately kept in the model.
The process does not degenerate there. Indeed, let $\ga=n-1$, so that
$\Ga(\ga+1)=(n-1)!$, and recall the convention $x_+^0=\mathbf 1_{\{x>0\}}$.
\begin{itemize}
\item If $u<\min\{0,t\}$, then $t-u>0$ and $-u>0$, and the binomial theorem
gives
\[
(t-u)^{n-1}-\sum_{k=0}^{n-1}\binom{n-1}{k}t^k(-u)^{n-1-k}=0 .
\]
\item If $0<u<t$, all the terms $(-u)_+^{n-1-k}$ vanish and
$f_n(t,u)=(t-u)^{n-1}/(n-1)!$.
\item If $t<u<0$, then $(t-u)_+=0$ and, by the binomial theorem again,
$f_n(t,u)=-(t-u)^{n-1}/(n-1)!$.
\item If $u>\max\{0,t\}$, then $f_n(t,u)=0$.
\end{itemize}
Hence, for $t\ge0$,
\[
X^{(n)}_H(t)=\int_0^t\frac{(t-u)^{n-1}}{(n-1)!}\,dM_\al(u),
\]
which is the Riemann-Liouville integral of order $n-1$ of the S$\al$S Lévy
motion $L(t)=M_\al([0,t])$; for $n=1$ this is $L$ itself, the case
$H=1/\al$ of linear fractional stable motion. The only special feature of
this value is that $\binom{\ga}{n}=0$, so the kernel has compact support and
the power tails used in Sections~\ref{sec:LRD} and~\ref{sec:increments}
disappear; this is why it is excluded there (conditions $a_s\neq0$,
$a_+\neq0$ and $a\neq0$), and it is treated separately in
Remark~\ref{rem:boundary} and Remark~\ref{rem:stationarity}. Sections~\ref{sec:definition}
and~\ref{sec:local} include it. Moreover the law depends continuously on
$H$ at this value, as everywhere in $(n-1,n)$: if $J$ is a compact interval
of $(n-1-\frac1\al,n-\frac1\al)$ and $g,g'\in J$, then, writing
$f_n(t,\cdot\,;g)-f_n(t,\cdot\,;g')=\int_{g'}^{g}\partial_\ga f_n(t,\cdot\,;\ga)\,d\ga$
and using Minkowski's integral inequality and Lemma~\ref{lem:deriv-bounds}
with $l=1$,
\[
\bigl\|f_n(t,\cdot\,;g)-f_n(t,\cdot\,;g')\bigr\|_{L^\al}
\le|g-g'|\,\sup_{\ga\in J}\bigl\|\partial_\ga f_n(t,\cdot\,;\ga)\bigr\|_{L^\al}
\le C_{t,J}\,|g-g'| .
\]
By \eqref{eq:fdd-integral}, the characteristic functions of the
finite-dimensional distributions are therefore continuous functions of
$(H(t_1),\ldots,H(t_m))$.
\end{remark}

\subsection{Harmonizable representation in the Gaussian case}
\label{subsec:harmonizable}

Let $W$ be a complex Gaussian random measure on $\R$ with Lebesgue control
measure which is Hermitian, $W(-A)=\overline{W(A)}$, and satisfies
$\E[W(A)\overline{W(B)}]=|A\cap B|$. For $h\in L^2(\R,\mathbb C)$ with
$h(-\xi)=\overline{h(\xi)}$, the integral $\int h\,dW$ is a real centred
Gaussian variable with variance $\int|h|^2$.

\begin{theorem}[Harmonizable representation, $\al=2$]
\label{thm:harmonizable}
Let $\al=2$ and let Assumption~\ref{ass:params} hold. For $t\in\R$ put
$E_n(z):=e^{iz}-\sum_{k=0}^{n-1}(iz)^k/k!$ and
\begin{equation}
\label{eq:harm-kernel}
\mathfrak g_n(t,\xi):=(-i\xi)^{-H(t)-1/2}\,E_n(-t\xi),\qquad\xi\neq0 .
\end{equation}
Then $\mathfrak g_n(t,\cdot)\in L^2(\R)$ is the Fourier transform of
$f_n(t,\cdot)$, and
\begin{equation}
\label{eq:harmonizable}
\bigl\{X^{(n)}_H(t)\bigr\}_{t\in\R}\overset{\mathrm{fdd}}{=}
\Bigl\{\frac{1}{\sqrt\pi}\int_\R\mathfrak g_n(t,\xi)\,W(d\xi)\Bigr\}_{t\in\R}.
\end{equation}
\end{theorem}

\begin{proof}
Fix $t$, put $\ga=\ga(t)=H(t)-\frac12$, and apply Lemma~\ref{lem:FT-power};
all Fourier transforms below are taken in the sense of tempered
distributions and computed on $\R\setminus\{0\}$. The function
$u\mapsto(t-u)_+^\ga$ is $y_+^\ga$ composed with $y=t-u$; hence its Fourier
transform coincides on $\xi\neq0$ with
\[
e^{-it\xi}\,\mathcal F[y_+^\ga](-\xi)=e^{-it\xi}\,\Ga(\ga+1)(-i\xi)^{-\ga-1}.
\]
Similarly, $u\mapsto(-u)_+^{\ga-k}$ has Fourier transform
$\Ga(\ga-k+1)(-i\xi)^{-\ga+k-1}$ (here $\ga-k\ge\ga-n+1>-\frac12$). Using
$\binom{\ga}{k}\Ga(\ga-k+1)=\Ga(\ga+1)/k!$, the Fourier transform of $f_n(t,\cdot)$
coincides on $\R\setminus\{0\}$ with
\begin{align*}
&\frac{1}{\Ga(\ga+1)}\Bigl[\Ga(\ga+1)e^{-it\xi}(-i\xi)^{-\ga-1}
-\sum_{k=0}^{n-1}\binom{\ga}{k}t^k\Ga(\ga-k+1)(-i\xi)^{-\ga-1+k}\Bigr]\\
&\qquad=(-i\xi)^{-\ga-1}\Bigl[e^{-it\xi}-\sum_{k=0}^{n-1}\frac{(-it\xi)^k}{k!}\Bigr]
=(-i\xi)^{-H(t)-1/2}E_n(-t\xi).
\end{align*}
This function is in $L^2(\R)$: $|E_n(z)|\le|z|^n/n!$ gives
$|\mathfrak g_n(t,\xi)|^2\le C|\xi|^{2(n-\ga-1)}$ near $0$, integrable since
$\ga<n-\frac12$; and $|E_n(z)|\le C(1+|z|^{n-1})$ gives
$|\mathfrak g_n(t,\xi)|^2\le C|\xi|^{2(n-\ga-2)}$ at infinity, integrable since
$\ga>n-\frac32$. Since $f_n(t,\cdot)\in L^2$ (Theorem~\ref{thm:existence}(i)
with $\al=2$), its Fourier transform is an
$L^2$ function; the difference between it and \eqref{eq:harm-kernel} is a
tempered distribution supported in $\{0\}$, hence a finite combination of
derivatives of the Dirac mass, and it belongs to $L^2$: it is therefore
zero.

Both sides of \eqref{eq:harmonizable} are centred real Gaussian processes
(the kernel is Hermitian because $f_n(t,\cdot)$ is real). By Plancherel,
\begin{align*}
\frac1\pi\int_\R\widehat f_n(s,\xi)\overline{\widehat f_n(t,\xi)}\,d\xi
&=\frac{2\pi}{\pi}\int_\R f_n(s,u)f_n(t,u)\,du\\
&=\mathrm{Cov}\bigl(X^{(n)}_H(s),X^{(n)}_H(t)\bigr),
\end{align*}
where the last equality is Lemma~\ref{lem:codiff-finite}(i),(iv). Since the
covariances coincide, so do the finite-dimensional distributions.
\end{proof}

\begin{remark}[The case $\al<2$]
\label{rem:harm-stable}
For $\al<2$ there is no $L^\al$ Parseval identity, and the process obtained
by replacing $W$ by a complex isotropic S$\al$S measure in
\eqref{eq:harmonizable} is not expected to coincide with $X^{(n)}_H$.
Already for $n=1$ and constant $H$, the linear and the harmonizable
fractional stable motions are not equal in law \cite{CambanisMaejima1989};
see also \cite[Chapter~7]{SamorTaqqu1994}.
\end{remark}

\subsection{Special cases}
\label{subsec:hierarchy}

If $H\equiv H_0$ is constant, $X^{(n)}_H$ is the $n$-FSM of \cite{Kawai2016}
(see Definition~\ref{def:nMFSM}); for $n=1$,
$X^{(1)}_{H_0}=\Ga(H_0+1-1/\al)^{-1}L_{H_0,\al}$. For $n=1$ and functional $H$,
\[
X^{(1)}_H(t)=\Ga\bigl(H(t)+1-1/\al\bigr)^{-1}L_H(t),
\]
where $L_H$ is the (causal) LMSM of \cite{StoevTaqqu2004}; the factor
depends on $t$, so $X^{(1)}_H$ is a deterministic, $t$-dependent multiple of
LMSM rather than LMSM itself. For $\al=2$, since
$\mathrm{Var}\,M_2(A)=2|A|$, $X^{(n)}_H$ coincides in finite-dimensional
distributions with $\sqrt2$ times the moving-average $n$-th order
multifractional Brownian motion of \cite[Definition~3]{GuptaKumarLeonenko2022}
on $[0,\infty)$, when $H$ satisfies the assumptions made there. For $\al=2$
and $n=1$ it is $\sqrt2$ times the moving-average multifractional Brownian
motion of \cite{PeltierLevyVehel1995}, which differs from the
harmonizable one of \cite{Benassi1997} when $H$ is not constant
\cite{StoevTaqqu2006}. For $\al=2$ and constant $H$ the kernel is the moving-average kernel of
the $n$-th order fractional Brownian motion of \cite{PerrinHarba2001}, so
the two processes differ by a deterministic factor $c>0$, depending on
$(n,H)$, which is fixed by comparing variances at $t=1$. By
Theorem~\ref{thm:harmonizable}, Plancherel's identity and the evenness of
$|E_n(-\xi)|^2$ in $\xi$,
\[
\mathrm{Var}\,X^{(n)}_H(1)=2\int_\R f_n(1,u)^2\,du
=\frac2\pi\int_0^\infty\xi^{-2H-1}\,|E_n(\xi)|^2\,d\xi ,
\]
so that, for instance, $c=\bigl(\mathrm{Var}\,X^{(n)}_H(1)\bigr)^{1/2}$ when the
process of \cite{PerrinHarba2001} is normalized to have unit variance at
$t=1$. For $n=1$ this gives the classical value
$\mathrm{Var}\,X^{(1)}_H(1)=-\frac4\pi\,\Ga(-2H)\cos(\pi H)$, $H\neq\frac12$.

To summarize the normalizations: the multiplicative factor between
$X^{(n)}_H$ and the cited processes is $1$ for the $n$-FSM of
\cite{Kawai2016}; $\Ga(H(t)+1-\frac1\al)^{-1}$ for the LMSM of
\cite{StoevTaqqu2004}; $\sqrt2$ for the processes of
\cite{GuptaKumarLeonenko2022} and \cite{PeltierLevyVehel1995}, whose kernels
coincide with ours and which are driven by a Brownian motion $B$ with
$\mathrm{Var}\,B(A)=|A|$, while $\mathrm{Var}\,M_2(A)=2|A|$; and the constant
$c$ above for \cite{PerrinHarba2001}.

\section{Local structure}
\label{sec:local}

In this section we show that the $n$-th order increments of $n$-MFSM behave
locally like those of an $n$-FSM. We use the $n$-th order increments
\eqref{eq:Delta-n} because they remove the polynomial part of the kernel
exactly when the exponent is frozen.

\begin{lemma}[Uniform bounds on the derivatives of the kernel]
\label{lem:deriv-bounds}
Let $T\subset\R$ be compact and $J\subset(n-1-1/\al,n-1/\al)$ a compact interval.
For integers $0\le i\le n-1$ and $0\le l\le n$,
\[
\sup_{t\in T,\ g\in J}\Bigl\|\partial_t^i\partial_g^l f_n(t,\cdot\,;g)\Bigr\|_{L^\al(\R)}<\infty,
\]
where the derivatives are taken pointwise for $u\notin\{t,0\}$. Moreover,
for $u\notin\{t,0\}$,
\begin{equation}
\label{eq:dtn}
\partial_t^n f_n(t,u;g)=\frac{(t-u)_+^{g-n}}{\Ga(g-n+1)} .
\end{equation}
Here and below $1/\Ga$ denotes the entire function $z\mapsto1/\Ga(z)$; in
particular \eqref{eq:dtn} vanishes when $g=n-1$.
\end{lemma}

\begin{proof}
Write $f_n(t,u;g)=\mathsf M(t-u;g)-\mathsf P(t,u;g)$ with
\[
\mathsf M(x;g):=\frac{x_+^g}{\Ga(g+1)},\qquad
\mathsf P(t,u;g):=\sum_{k=0}^{n-1}\frac{t^k\,(-u)_+^{g-k}}{k!\,\Ga(g-k+1)},
\]
where we used $\binom{g}{k}/\Ga(g+1)=1/(k!\,\Ga(g-k+1))$, an identity which
holds for all $k\ge0$ and $g>-1$ with the above convention on $1/\Ga$. Note that
$g-k+1\ge g-n+2>1-1/\al>0$, so all Gamma factors are smooth and positive in
$g\in J$.

\emph{Formula \eqref{eq:dtn}.} $\mathsf P$ is a polynomial of degree $n-1$
in $t$, so $\partial_t^n\mathsf P=0$, and
$\partial_t^n\mathsf M(t-u;g)=g(g-1)\cdots(g-n+1)(t-u)_+^{g-n}/\Ga(g+1)
=(t-u)_+^{g-n}/\Ga(g-n+1)$.

\emph{Bounded region.} Let $R:=2\sup_{t\in T}|t|+1$. For $u\in(-R,R)$,
\[
\partial_t^i\mathsf M(t-u;g)=\frac{(t-u)_+^{g-i}}{\Ga(g-i+1)},\qquad
\partial_t^i\mathsf P(t,u;g)=\sum_{k=i}^{n-1}\frac{t^{k-i}(-u)_+^{g-k}}{(k-i)!\,\Ga(g-k+1)} .
\]
Each $\partial_g$ produces at most one factor $\log(t-u)$ or $\log(-u)$ and
derivatives of $1/\Ga$, which are bounded on $J$. Hence
\[
\bigl|\partial_t^i\partial_g^lf_n(t,u;g)\bigr|
\le C\Bigl[(t-u)_+^{g-i}+\sum_{k=0}^{n-1}(-u)_+^{g-k}\Bigr]\bigl(1+|\log|t-u||+|\log|u||\bigr)^l .
\]
Since $x^{g}\le x^{\min J}+x^{\max J}$ for $x>0$, all exponents lie in
$[\al(\min J-n+1),\,\al\max J]$ with $\al(\min J-n+1)>-1$, and
$\sup_{t\in T}\int_{-R}^{R}|x|^{-p}(1+|\log|x-t||+|\log|x||)^{q}dx<\infty$ for
$p<1$, $q\ge0$ (this includes $t=0$, where the two singularities merge).
Hence the $\al$-th power of the right-hand side is integrable on $(-R,R)$,
uniformly in $(t,g)\in T\times J$.

\emph{Left tail.} For $u=-v$ with $v\ge R\ge2|t|+1$,
Lemma~\ref{lem:kernel-asymp}(ii) (valid for every $g>-1$) gives the
absolutely convergent expansion
$f_n(t,-v;g)=\sum_{k\ge n}a_k(g)\,t^kv^{g-k}$ with $a_k(g):=1/(k!\,\Ga(g-k+1))$.
We bound the $g$-derivatives of $a_k$. For $k\ge n+2$, $z_0=g-k+1$ and
$|w-z_0|\le1$, the reflection formula gives
$1/\Ga(w)=\Ga(1-w)\sin(\pi w)/\pi$. Since
$\Re(1-w)\in[k-\max J-1,\,k-\min J+1]\subset[1+\tfrac1\al,\,k+2]$, we have
$|\Ga(1-w)|\le\Ga(\Re(1-w))\le\Ga(k+2)=(k+1)\,k!$, and
$|\sin(\pi w)|\le\cosh\pi$. By Cauchy's estimate on the disc of radius $1$,
$|(1/\Ga)^{(m)}(z_0)|\le m!\,\pi^{-1}\cosh(\pi)\,(k+1)\,k!$, hence
$|\partial_g^m a_k(g)|\le C(k+1)$ for $m\le n$, uniformly on $J$; the finitely
many $k\le n+1$ are bounded by continuity. Each term
$a_k(g)t^kv^{g-k}$ differentiated $i$ times in $t$ and $l$ times in $g$ is
therefore bounded by
\[
C(k+1)k^i|t|^{k-i}v^{g-k}(1+\log v)^l
=C(k+1)k^i|t|^{n-i}\,v^{g-n}(1+\log v)^l\Bigl(\frac{|t|}{v}\Bigr)^{k-n},
\]
which is summable in $k$ since $|t|/v\le\frac12$. By the Weierstrass M-test
we may differentiate term by term, and for $v\ge R$,
\[
\bigl|\partial_t^i\partial_g^lf_n(t,-v;g)\bigr|
\le C\,v^{g-n}(1+\log v)^l\sum_{k\ge n}(k+1)k^{i}2^{-(k-n)}
\le C'\,v^{g-n}(1+\log v)^l .
\]
Since $\al(g-n)\le\al(\max J-n)<-1$, this is in $L^\al(R,\infty)$ uniformly.

\emph{Right side.} For $u\ge R$, $f_n(t,u;g)=0$ by
Lemma~\ref{lem:kernel-asymp}(i), for all $t\in T$.
\end{proof}

\begin{theorem}[Local asymptotic self-similarity of the $n$-th order increments]
\label{thm:LASS}
Let Assumption~\ref{ass:params} hold, let $t_0\in\R$, and assume that $H$ is
of class $\mathcal C^n$ on an open neighbourhood $V$ of $t_0$ (so that
$H',\ldots,H^{(n)}$ are bounded on every compact subset of $V$). Put $H_0:=H(t_0)$ and let
$Y=Y^{(n)}_{H_0}$ be the $n$-FSM with index $H_0$. Then
\begin{equation}
\label{eq:LASS}
\Bigl\{\eps^{-H_0}\,\De^n_{\eps\tau}X^{(n)}_H(t_0)\Bigr\}_{\tau\in\R}
\overset{\mathrm{fdd}}{\longrightarrow}
\Bigl\{\De^n_{\tau}Y(0)\Bigr\}_{\tau\in\R},\qquad\eps\downarrow0 .
\end{equation}
\end{theorem}

\begin{proof}
Write $\ga_0:=\ga(t_0)$, $t_j:=t_0+jh$, $\ga_j:=\ga(t_j)$ for a step
$h\in\R\setminus\{0\}$.
By linearity, the kernel of $\De^n_hX^{(n)}_H(t_0)$ is
$\sum_jc_jf_n(t_j,\cdot\,;\ga_j)=k^{(0)}_h+K_h$ with
\[
k^{(0)}_h(u):=\sum_{j=0}^nc_jf_n(t_j,u;\ga_0),\qquad
K_h(u):=\sum_{j=0}^nc_j\bigl[f_n(t_j,u;\ga_j)-f_n(t_j,u;\ga_0)\bigr].
\]

\medskip\noindent\textbf{Step 1: the frozen part.}
Since $\sum_jc_j(t_0+jh)^k=0$ for $k\le n-1$ (the $n$-th difference of a
polynomial of degree $<n$), the polynomial terms of \eqref{eq:kernel} cancel
in $k^{(0)}_h$:
\[
k^{(0)}_h(u)=\frac{1}{\Ga(\ga_0+1)}\sum_{j=0}^nc_j(jh-(u-t_0))_+^{\ga_0}
=\sum_{j=0}^nc_jf_n(jh,u-t_0;\ga_0),
\]
where the last equality uses the same cancellation with $t_0$ replaced by
$0$. By \eqref{eq:translation}, jointly in finitely many $\tau$'s,
$\bigl(I(k^{(0)}_{\eps\tau})\bigr)_\tau\overset{d}{=}\bigl(\De^n_{\eps\tau}Y(0)\bigr)_\tau$, and by
the self-similarity of Theorem~\ref{thm:existence}(iii),
$\bigl(\eps^{-H_0}\De^n_{\eps\tau}Y(0)\bigr)_\tau\overset{d}{=}\bigl(\De^n_\tau Y(0)\bigr)_\tau$. Thus
the frozen part has, for every $\eps$, exactly the law of the limit.

\medskip\noindent\textbf{Step 2: reduction.}
We show below that
\begin{equation}
\label{eq:Kh-small}
\|K_h\|_{L^\al}=o(|h|^{H_0})\qquad(h\to0).
\end{equation}
Granting this, fix $\tau_1,\ldots,\tau_m$ and $\theta\in\R^m$, and put
$A_\eps:=\eps^{-H_0}\sum_i\theta_ik^{(0)}_{\eps\tau_i}$ and
$B_\eps:=\eps^{-H_0}\sum_i\theta_iK_{\eps\tau_i}$ (for $\tau_i=0$ both kernels
vanish). By \eqref{eq:Kh-small},
\[
\|B_\eps\|_{L^\al}\le\sum_{i:\tau_i\neq0}|\theta_i|\,|\tau_i|^{H_0}\,
\frac{\|K_{\eps\tau_i}\|_{L^\al}}{|\eps\tau_i|^{H_0}}\longrightarrow0 .
\]
By Step~1, $\|A_\eps\|_{L^\al}$ does not
depend on $\eps$. Since
$\bigl|\|A_\eps+B_\eps\|_{L^\al}-\|A_\eps\|_{L^\al}\bigr|\le\|B_\eps\|_{L^\al}$,
\eqref{eq:fdd-integral} shows that the characteristic function of
$\bigl(\eps^{-H_0}\De^n_{\eps\tau_i}X^{(n)}_H(t_0)\bigr)_i$ at $\theta$ converges to
that of $(\De^n_{\tau_i}Y(0))_i$, which proves \eqref{eq:LASS}.

\medskip\noindent\textbf{Step 3: proof of \eqref{eq:Kh-small}.}
Let $U=[t_0-\eta,t_0+\eta]$, $\eta>0$, be such that $H$ is of class
$\mathcal C^n$ on a neighbourhood of $U$, and let $L$ be a Lipschitz constant
of $\ga$ on $U$. Fix a compact interval
$J\subset(n-1-\frac1\al,n-\frac1\al)$ containing $\ga(U)$ in its interior
(possible since $\ga(U)\subset[a_H-\frac1\al,b_H-\frac1\al]$). Put
$\delta:=|h|$ and assume $\delta$ so small that
\[
n\delta\le\eta,\quad 4n\delta\le1,\quad \ga_0\pm Ln\delta\in J,\quad
\al(\ga_0+Ln\delta-n)<-1,\quad \al(\ga_0-Ln\delta-n+1)>-1 .
\]
Then $|\ga_j-\ga_0|\le Ln\delta$ for all $j$. We split $\R$ into
\[
N:=[t_0-3n\delta,\;t_0+n\delta]\qquad\text{and}\qquad F:=\R\setminus N .
\]
All $t_j$ lie in $[t_0-n\delta,t_0+n\delta]$.

\smallskip\noindent\emph{(a) The region $F$.}
Fix $u\in F\setminus\{0\}$ and let
$\varphi(s):=f_n(t_0+s,u;\ga(t_0+s))-f_n(t_0+s,u;\ga_0)$ for $|s|\le n\delta$.
For such $s$, $t_0+s\neq u$, so $\varphi\in\mathcal C^n$, and $K_h(u)=\De^n_h\varphi(0)$.
Iterating $\De_h\phi(s)=h\int_0^1\phi'(s+h\vartheta)d\vartheta$ gives
\[
K_h(u)=h^n\int_{[0,1]^n}\varphi^{(n)}\bigl(h(\vartheta_1+\cdots+\vartheta_n)\bigr)d\vartheta,
\]
hence
\[
\|K_h\mathbf 1_F\|_{L^\al}\le\delta^n\sup_{|s|\le n\delta}\|\varphi^{(n)}(s)\mathbf 1_F\|_{L^\al}
\]
by Minkowski's integral inequality. By Faà di Bruno's formula, with
$t=t_0+s$,
\[
\varphi^{(n)}(s)=\Bigl[\partial_t^nf_n(t,u;\ga(t))-\partial_t^nf_n(t,u;\ga_0)\Bigr]
+\sum a_{i,l,r}\,\partial_t^{i}\partial_g^{l}f_n(t,u;\ga(t))\prod_{q=1}^{l}\ga^{(r_q)}(t),
\]
where the finite sum runs over $l\ge1$, $r_q\ge1$ and $i+\sum_qr_q=n$ (so
$i\le n-1$), with integer coefficients $a_{i,l,r}$ independent of $u$. By Lemma~\ref{lem:deriv-bounds} and the boundedness of
$\ga',\ldots,\ga^{(n)}$ on $U$, the $L^\al$ norm of the sum is bounded by a
constant. For the bracket, \eqref{eq:dtn} and the mean value theorem in $g$
give, for $u<t$,
\[
\Bigl|\partial_t^nf_n(t,u;\ga(t))-\partial_t^nf_n(t,u;\ga_0)\Bigr|
\le C\,\delta\,(t-u)^{g_*-n}\bigl(1+|\log(t-u)|\bigr),
\]
with $g_*=g_*(t-u)$ equal to $\ga_0-Ln\delta$ if $t-u\le1$ and to
$\ga_0+Ln\delta$ if $t-u>1$;
for $u>t$ the bracket vanishes. For $u\in F$ with $u<t$ we have
$u<t_0-3n\delta$, hence $t-u\ge2n\delta$. Therefore
\[
\Bigl\|\bigl[\cdots\bigr]\mathbf 1_F\Bigr\|_{L^\al}^\al
\le C\delta^\al\Bigl[\int_{2n\delta}^{1}x^{\al(\ga_0-Ln\delta-n)}(1+|\log x|)^\al dx
+\int_1^\infty x^{\al(\ga_0+Ln\delta-n)}(1+\log x)^\al dx\Bigr].
\]
For $\delta$ small, $\al(\ga_0+Ln\delta-n)<-1$, so the second integral is
bounded; the first is at most
$C\delta^{\al(\ga_0-Ln\delta-n)+1}(1+|\log\delta|)^\al$, and
$\delta^{-Ln\delta\al}\le C$. Hence
\[
\|K_h\mathbf 1_F\|_{L^\al}
\le\delta^n\Bigl[C+C\delta\cdot\delta^{\ga_0-n+1/\al}(1+|\log\delta|)\Bigr]
=C\delta^n+C\delta^{1+H_0}(1+|\log\delta|)=O(\delta^n),
\]
because $\ga_0+1/\al=H_0$ and $1+H_0>n$; in particular this is
$o(\delta^{H_0})$ since $H_0<n$.

\smallskip\noindent\emph{(b) The region $N$.}
With the notation $\mathsf M,\mathsf P$ of the proof of
Lemma~\ref{lem:deriv-bounds}, $K_h=K_h^{\mathsf M}-K_h^{\mathsf P}$, where
$K_h^{\mathsf M}(u):=\sum_jc_j[\mathsf M(t_j-u;\ga_j)-\mathsf M(t_j-u;\ga_0)]$ and
similarly for $\mathsf P$.

For $\mathsf M$: by the mean value theorem in $g$, for $0<x\le4n\delta\le1$,
$|\mathsf M(x;\ga_j)-\mathsf M(x;\ga_0)|\le C\delta\,x^{\ga_0-Ln\delta}(1+|\log x|)$,
and $\al(\ga_0-Ln\delta)>-1$ since $\ga_0>-1/\al$ and $\delta$ is small.
Since $t_j-u\le4n\delta$ on $N$,
\begin{align*}
\|K_h^{\mathsf M}\mathbf 1_N\|_{L^\al}
&\le C\delta\sum_{j=0}^n\Bigl(\int_0^{4n\delta}x^{\al(\ga_0-Ln\delta)}(1+|\log x|)^\al dx\Bigr)^{1/\al}\\
&\le C\delta\cdot\delta^{\ga_0+1/\al}(1+|\log\delta|)
=C\delta^{1+H_0}(1+|\log\delta|).
\end{align*}
For $\mathsf P$: for $u\in N\setminus\{0\}$, the function
$\phi(s):=\mathsf P(t_0+s,u;\ga(t_0+s))-\mathsf P(t_0+s,u;\ga_0)$ is $\mathcal C^n$ in
$s$, and $K_h^{\mathsf P}(u)=\De^n_h\phi(0)$. As in (a), and since
$\partial_t^n\mathsf P=0$, $\phi^{(n)}$ is a finite sum of terms
$\partial_t^i\partial_g^l\mathsf P(t,u;\ga(t))\prod_q\ga^{(r_q)}(t)$ with $l\ge1$, each
bounded by $C_{t_0}\sum_{k\le n-1}(-u)_+^{\ga_0-Ln\delta-k}(1+|\log|u||)^n$, where
$C_{t_0}$ absorbs the factor $(\sup_N|u|)^{2Ln\delta}$ needed when $|u|>1$;
its $\al$-th power is integrable on $N$ because $\al(\ga_0-Ln\delta-n+1)>-1$. Hence
$\|K_h^{\mathsf P}\mathbf 1_N\|_{L^\al}\le C\delta^n$.

Combining (a) and (b),
$\|K_h\|_{L^\al}\le C\delta^n+C\delta^{1+H_0}(1+|\log\delta|)=o(\delta^{H_0})$, which is
\eqref{eq:Kh-small}.
\end{proof}

\begin{remark}[The case $n=1$ and the regularity of $H$]
\label{rem:LASS-n1}
For $n=1$, $X^{(1)}_H(t)=\Ga(H(t)+1-1/\al)^{-1}L_H(t)$ where $L_H$ is the LMSM
of \cite{StoevTaqqu2004}, and Theorem~\ref{thm:LASS} is contained in
\cite[Theorem~5.1]{StoevTaqqu2004}: Stoev and Taqqu prove the local
asymptotic self-similarity of LMSM at $t_0\neq0$ under
$H(t)-H(t_0)=o(|t-t_0|^{H(t_0)})$, which holds when $H$ is differentiable at
$t_0$, and at $t_0=0$ under $H(t)-H(0)=o(1/\ln|t|)$. The point of
Theorem~\ref{thm:LASS} is the case $n\ge2$. For a related model
in which $H$ is evaluated at the integration variable, Mies and Sikkens
\cite{MiesSikkens2026} identify the tangent process under a mere modulus of
continuity.

For $n=1$ our hypothesis ($H$ of class $\mathcal C^1$) is stronger than the
condition $H(t)-H(t_0)=o(|t-t_0|^{H(t_0)})$ of Stoev and Taqqu. In the proof,
the regularity of $H$ enters at two places only. In Step~3(b), the part
$K_h^{\mathsf M}$ uses only that $H$ is Lipschitz near $t_0$. In Step~3(a)
and in the part $K_h^{\mathsf P}$ of Step~3(b), it enters through the
estimate
\[
|\De^n_h\varphi(0)|\le|h|^n\sup_{|s|\le n|h|}|\varphi^{(n)}(s)|
\]
for the functions $\varphi$, $\phi$ obtained by composing $\ga$ with the smooth
maps $g\mapsto f_n(t,u;g)$, $g\mapsto\mathsf P(t,u;g)$, together with the
$L^\al(du)$ bounds of Lemma~\ref{lem:deriv-bounds}. This is the only
estimate that requires $n$ derivatives of $H$. Since only
$\|K_h\|_{L^\al}=o(|h|^{H(t_0)})$ is needed and $H(t_0)<n$, it would suffice to
have $|\De^n_h\varphi(0)|\le C|h|^{r}$ with some $r>H(t_0)$ and a constant
$C=C(u)$ whose $\al$-th power is integrable in $u$. For a Hölder-Zygmund
function $H$ of order $r\in(H(t_0),n)$, such a bound holds for each fixed
$u$, but we have not proved the required integrability of $C(u)$ near the
singular points $u\in\{0,t_0\}$. We therefore do not know whether
Theorem~\ref{thm:LASS} holds under this weaker assumption when $n\ge2$.
\end{remark}

\begin{remark}[Pointwise regularity]
\label{rem:holder}
For LFSM with $\al\in[1,2)$ and $H\in(1/\al,1)$, the local Hölder exponent
equals $H-1/\al$ and the pointwise Hölder exponent at a fixed time lies
almost surely in $[H-1/\al,H]$ (results of Takashima and of Kôno and Maejima,
recalled in \cite[Section~1.4]{Balanca2014}); moreover the multifractal
spectrum is $h\mapsto\al(h-H)+1$ on $[H-1/\al,H]$
\cite[Corollary~1.11]{Balanca2014}. For LMSM with $H$ taking values in
$(1/\al,1)$ and sufficiently regular, the local Hölder exponent at $t$ equals
$H(t)-1/\al$ almost surely \cite{AyacheHamonier2014}. For constant $H$ and
$n\ge2$, the derivative of the $n$-FSM is the $(n-1)$-FSM with index $H-1$
\cite[Theorem~4.1(vi)]{Kawai2016}, so its regularity reduces to that of LFSM
with index $H-n+1$. For $n\ge2$ (where
$H(t)>1/\al$ automatically) we expect the local Hölder exponent of
$X^{(n)}_H$, defined through $n$-th order differences because $H(t)-1/\al$
may exceed $1$, to be $H(t)-1/\al$. Theorem~\ref{thm:LASS} is consistent with
this but does not imply it. We do not prove this statement here; a proof
would require uniform moduli of continuity for stable integrals of the type
obtained in \cite{AyacheHamonier2014}. When $\al=2$ and $H$ is sufficiently regular, both
exponents are expected to equal $H(t)$, as for multifractional Brownian
motion.
\end{remark}

\section{Asymptotic codifference of the process}
\label{sec:LRD}

\subsection{Setting and notation}

Throughout this section, $\al\in(1,2]$, $n\in\N^*$ and $H$ satisfies
Assumption~\ref{ass:params}. We recall that $\ga(t)=H(t)-1/\al$ and that, for every
$t\in\R$, $X^{(n)}_H(t)=\int_\R f_n(t,u)\,dM_\al(u)$ with
\[
f_n(t,u)=\frac{1}{\Ga(\ga(t)+1)}\Bigl[(t-u)_+^{\ga(t)}
-\sum_{k=0}^{n-1}\binom{\ga(t)}{k}t^k(-u)_+^{\ga(t)-k}\Bigr].
\]
We write $f_n(t,u;\ga)$ when we want to stress the value $\ga=\ga(t)$ of
the exponent. Since $f_n(t,\cdot)$ depends on the function $H$ only through
the number $H(t)$, the codifference
\[
\tau(s,t):=\tau\bigl(X^{(n)}_H(s),X^{(n)}_H(t)\bigr)
=\int_\R\Phi\bigl(f_n(s,u),f_n(t,u)\bigr)\,du,
\]
with $\Phi(a,b):=|a|^\al+|b|^\al-|a-b|^\al$,
depends on $H$ only through $H(s)$ and $H(t)$.

For a real number $\ga$ we set
\begin{equation}
\label{eq:a-def}
a(\ga):=\ga-(n-1),
\end{equation}
and we note that $\ga\in(n-1-1/\al,\,n-1/\al)$ if and only if
$a(\ga)\in(-1/\al,\,1-1/\al)\subset(-1,1)$. For $w>0$ we define
\begin{equation}
\label{eq:phi-def}
\varphi_{n,\ga}(w):=(1+w)^{\ga}-\sum_{k=0}^{n-1}\binom{\ga}{k}w^{\ga-k},
\qquad
F_\ga(w):=\frac{\varphi_{n,\ga}(w)}{\Ga(\ga+1)} .
\end{equation}
Finally, $\psi(x):=|x|^{\al-1}\sign(x)$ for $x\in\R$; note that
$\frac{d}{dx}|x|^\al=\al\,\psi(x)$ and that $\psi(\lambda x)=\lambda^{\al-1}\psi(x)$
for $\lambda>0$.

\emph{Role of the assumption $\al>1$.} In this section $\al>1$ is used in
three places only: (i) $x\mapsto|x|^\al$ is then of class $\mathcal C^1$ with
derivative $\al\psi$, and $\psi$ is Hölder continuous of order $\al-1\in(0,1]$,
which gives the linearization of Lemma~\ref{lem:ineq} and Step~1 of the
proof of Theorem~\ref{thm:codiff-asymp}; (ii) $|\psi(h)|=|h|^{\al-1}$ belongs
to $L^{\al/(\al-1)}$ when $h\in L^\al$, so that the main term $L_t$ is
absolutely convergent by Hölder's inequality; (iii) $\ga(t)>-\frac1\al>-1$, so
that all Gamma factors and the expansions of Lemma~\ref{lem:phi} are
available. For $\al\le1$ the function $|x|^\al$ is not differentiable at $0$
and a different decomposition of $\Phi$ would be needed.

The notation of this section is summarized in the following table.
\begin{center}
\small
\begin{tabular}{@{}lll@{}}
\toprule
Symbol & Meaning & Defined in\\
\midrule
$\ga(t)$, $a(\ga)$ & $H(t)-\frac1\al$, and $\ga-(n-1)$ & \eqref{eq:a-def}\\
$\varphi_{n,\ga}$, $F_\ga$ & scaled kernel on $(-\infty,0)$, and $\varphi_{n,\ga}/\Ga(\ga+1)$ & \eqref{eq:phi-def}\\
$\psi$, $\Phi$, $R$ & $|x|^{\al-1}\sign x$; $\Phi(a,b)=|a|^\al+|b|^\al-|a-b|^\al$; remainder & Lemma~\ref{lem:ineq}\\
$\omega_a$ & profile of $|\varphi_{n,\ga}|$ & \eqref{eq:omega}\\
$\ga_s,\ga_+$, $a_s,a_+$ & $\ga(s)$, $\lim\ga(t)$, and $a(\ga_s)$, $a(\ga_+)$ & Subsection~\ref{subsec:main5}\\
$C_n(s)$, $\beta_+$, $\Ga_+$ & tail coefficient of $f_n(s,\cdot)$; $\binom{\ga_+}{n-1}/\Ga_+$; $\Ga(\ga_++1)$ & Subsection~\ref{subsec:main5}\\
$c_0(s)$ & $\|X^{(n)}_H(s)\|_\al^\al$ if $\al<2$, $0$ if $\al=2$ & Subsection~\ref{subsec:main5}\\
$\rho(s)$, $C_1,C_2,C_4$ & exponent and constants of the three regimes & \eqref{eq:C1}--\eqref{eq:rho-unified}\\
$K(s,p)$ & $\int_0^\infty v^{p}f_n(s,-v)\,dv$, closed form & Lemma~\ref{lem:mellin}\\
$D(s)$, $\tau_\infty(s)$ & correction for $n=1$, $\al<2$ in E4; limit in E2 & Theorem~\ref{thm:codiff-asymp}\\
\bottomrule
\end{tabular}
\end{center}

\begin{lemma}[Scaling identity]
\label{lem:scaling}
Let $t>0$ and $\ga=\ga(t)$. Then
\begin{align}
f_n(t,u)&=\frac{(t-u)^{\ga}}{\Ga(\ga+1)}, && 0<u<t,
\label{eq:f-pos}\\
f_n(t,u)&=0, && u>t,
\label{eq:f-zero}\\
f_n(t,-v)&=t^{\ga}\,F_\ga(v/t), && v>0.
\label{eq:f-neg}
\end{align}
\end{lemma}

\begin{proof}
For $u>0$ all the terms $(-u)_+^{\ga-k}$ vanish, which gives
\eqref{eq:f-pos} and \eqref{eq:f-zero}. For $u=-v<0$ we have
$(t-u)_+=t+v$ and $(-u)_+=v$, hence
\begin{align*}
\Ga(\ga+1)\,f_n(t,-v)
&=(t+v)^{\ga}-\sum_{k=0}^{n-1}\binom{\ga}{k}t^k v^{\ga-k}\\
&=t^{\ga}\Bigl[\Bigl(1+\frac vt\Bigr)^{\ga}
-\sum_{k=0}^{n-1}\binom{\ga}{k}\Bigl(\frac vt\Bigr)^{\ga-k}\Bigr]
=t^{\ga}\,\varphi_{n,\ga}(v/t),
\end{align*}
since $t^k v^{\ga-k}=t^{\ga}(v/t)^{\ga-k}$.
\end{proof}

\subsection{Five auxiliary lemmas}

\begin{lemma}[Integral representation]
\label{lem:repr}
Let $\ga\in\R$. For every $w>0$,
\begin{equation}
\label{eq:repr}
\varphi_{n,\ga}(w)=n\binom{\ga}{n}\int_0^1(1-\theta)^{n-1}(w+\theta)^{\ga-n}\,d\theta .
\end{equation}
\end{lemma}

\begin{proof}
Let $q(x):=(1+x)^{\ga}$ for $x>-1$. Its derivatives are
$q^{(k)}(x)=\ga(\ga-1)\cdots(\ga-k+1)(1+x)^{\ga-k}=k!\binom{\ga}{k}(1+x)^{\ga-k}$.
Taylor's formula of order $n$ at $0$ with integral remainder gives, for $x>0$,
\begin{align*}
(1+x)^{\ga}
&=\sum_{k=0}^{n-1}\frac{q^{(k)}(0)}{k!}x^k
+\frac{x^n}{(n-1)!}\int_0^1(1-\theta)^{n-1}q^{(n)}(\theta x)\,d\theta\\
&=\sum_{k=0}^{n-1}\binom{\ga}{k}x^k
+\frac{x^n}{(n-1)!}\,n!\binom{\ga}{n}\int_0^1(1-\theta)^{n-1}(1+\theta x)^{\ga-n}\,d\theta\\
&=\sum_{k=0}^{n-1}\binom{\ga}{k}x^k
+n\binom{\ga}{n}x^n\int_0^1(1-\theta)^{n-1}(1+\theta x)^{\ga-n}\,d\theta .
\end{align*}
We take $x=1/w$ and multiply by $w^{\ga}$. Since
$w^{\ga}(1+1/w)^{\ga}=(1+w)^{\ga}$, $w^{\ga}w^{-k}=w^{\ga-k}$ and
\[
w^{\ga}\,w^{-n}\Bigl(1+\frac{\theta}{w}\Bigr)^{\ga-n}
=w^{\ga-n}\Bigl(\frac{w+\theta}{w}\Bigr)^{\ga-n}
=(w+\theta)^{\ga-n},
\]
we obtain
$(1+w)^{\ga}-\sum_{k=0}^{n-1}\binom{\ga}{k}w^{\ga-k}
=n\binom{\ga}{n}\int_0^1(1-\theta)^{n-1}(w+\theta)^{\ga-n}\,d\theta$,
which is \eqref{eq:repr}.
\end{proof}

In the rest of the section, for $a\in(-1,1)\setminus\{0\}$, we use the profile
\begin{equation}
\label{eq:omega}
\omega_a(w):=
\begin{cases}
(1+w)^{a-1}, & a>0,\\[2pt]
w^{a}(1+w)^{-1}, & a<0,
\end{cases}
\qquad w>0 .
\end{equation}

\begin{lemma}[Properties of $\varphi_{n,\ga}$]
\label{lem:phi}
Let $\ga\in(n-1-1/\al,n-1/\al)$ with $a=a(\ga)\neq0$. Then:
\begin{enumerate}[label=(\roman*)]
\item \emph{Signs.} $\sign\binom{\ga}{n}=\sign(a)$ and $\binom{\ga}{n-1}>0$
(with $\binom{\ga}{0}=1$). Consequently $\varphi_{n,\ga}(w)$ has the sign of
$a$ for every $w>0$, and $\varphi_{n,\ga}$ never vanishes on $(0,\infty)$.
\item \emph{Monotonicity.} $|\varphi_{n,\ga}|$ is strictly decreasing on
$(0,\infty)$.
\item \emph{Two-sided bound.} There exist $0<c\le C<\infty$ such that
$c\,\omega_a(w)\le|\varphi_{n,\ga}(w)|\le C\,\omega_a(w)$ for all $w>0$.
\item \emph{Behaviour at infinity.} For $w\ge1$,
\begin{equation}
\label{eq:phi-infty}
\Bigl|\varphi_{n,\ga}(w)-\binom{\ga}{n}w^{a-1}\Bigr|
\le\frac{1-a}{n+1}\Bigl|\binom{\ga}{n}\Bigr|\,w^{a-2}.
\end{equation}
\item \emph{Behaviour at zero.} There is $M<\infty$ such that, for $0<w\le1$,
\begin{align}
\Bigl|\varphi_{n,\ga}(w)+\binom{\ga}{n-1}w^{a}\Bigr|&\le M, && \text{if } a<0,
\label{eq:phi-zero-neg}\\
\Bigl|\varphi_{n,\ga}(w)-1+\binom{\ga}{n-1}w^{a}\Bigr|&\le M\,w, && \text{if } a>0.
\label{eq:phi-zero-pos}
\end{align}
\end{enumerate}
All constants depend continuously on $\ga$; in particular they can be chosen
uniformly for $\ga$ in any compact subset of
$\{\ga\in(n-1-1/\al,n-1/\al):a(\ga)\neq0\}$.
\end{lemma}

\begin{proof}
\emph{(i)} We have $n!\binom{\ga}{n}=\prod_{j=0}^{n-1}(\ga-j)$. For
$0\le j\le n-2$, $\ga-j\ge\ga-n+2=a+1>0$, and the last factor is
$\ga-(n-1)=a$. Hence $\sign\binom{\ga}{n}=\sign(a)$. In the same way
$(n-1)!\binom{\ga}{n-1}=\prod_{j=0}^{n-2}(\ga-j)>0$. Since the integral in
\eqref{eq:repr} is strictly positive, $\sign\varphi_{n,\ga}(w)=\sign\binom{\ga}{n}=\sign(a)$.

\emph{(ii)} Differentiating \eqref{eq:repr} under the integral sign (the
integrand is smooth in $w>0$ and dominated on compact subsets of $(0,\infty)$),
\[
\varphi_{n,\ga}'(w)=n\binom{\ga}{n}(\ga-n)\int_0^1(1-\theta)^{n-1}(w+\theta)^{\ga-n-1}\,d\theta .
\]
Since $\ga-n=a-1<0$, $\varphi'_{n,\ga}$ has the sign opposite to
$\varphi_{n,\ga}$, so $|\varphi_{n,\ga}|$ is strictly decreasing.

\emph{(iii)} Write $I(w):=\int_0^1(1-\theta)^{n-1}(w+\theta)^{a-1}d\theta$, so that
$|\varphi_{n,\ga}(w)|=n\bigl|\binom{\ga}{n}\bigr|I(w)$. Recall $a-1<0$.

\smallskip\noindent
\emph{Case $w\ge1$.} Since $w\le w+\theta\le w+1\le 2w$,
\[
\frac{2^{a-1}}{n}\,w^{a-1}
\le(w+1)^{a-1}\int_0^1(1-\theta)^{n-1}d\theta
\le I(w)
\le w^{a-1}\int_0^1(1-\theta)^{n-1}d\theta=\frac{w^{a-1}}{n}.
\]
On $[1,\infty)$ one has $\omega_a(w)\le w^{a-1}\le 2^{1-a}\,\omega_a(w)$ if $a>0$
(because $w\le1+w\le2w$) and $\omega_a(w)\le w^{a-1}\le 2\,\omega_a(w)$ if $a<0$.
This gives the bound for $w\ge1$.

\smallskip\noindent
\emph{Case $0<w<1$.} We use $(1-\theta)^{n-1}\le1$ on $[0,1]$ and
$(1-\theta)^{n-1}\ge2^{1-n}$ on $[0,1/2]$:
\[
2^{1-n}\int_0^{1/2}(w+\theta)^{a-1}d\theta\le I(w)\le\int_0^1(w+\theta)^{a-1}d\theta .
\]
For $c>0$, $\int_0^{c}(w+\theta)^{a-1}d\theta=\frac{(w+c)^{a}-w^{a}}{a}$.

If $a>0$, the upper bound is $\frac{(w+1)^a-w^a}{a}\le\frac{2^a}{a}$. The map
$w\mapsto(w+\frac12)^a-w^a$ is decreasing (its derivative is
$a[(w+\frac12)^{a-1}-w^{a-1}]<0$), so on $(0,1)$ the lower bound is at least
$2^{1-n}\frac{(3/2)^a-1}{a}>0$. Hence $I(w)\asymp1\asymp\omega_a(w)$.

If $a<0$, the upper bound is
$\frac{w^a-(w+1)^a}{|a|}\le\frac{w^a}{|a|}$, and
\[
\int_0^{1/2}(w+\theta)^{a-1}d\theta
=\frac{w^{a}-(w+\frac12)^{a}}{|a|}
=\frac{w^{a}}{|a|}\Bigl[1-\Bigl(1+\frac1{2w}\Bigr)^{a}\Bigr]
\ge\frac{w^{a}}{|a|}\Bigl[1-\Bigl(\frac32\Bigr)^{a}\Bigr],
\]
because $1+\frac1{2w}\ge\frac32$ and $x\mapsto x^{a}$ is decreasing. Hence
$I(w)\asymp w^a\asymp\omega_a(w)$ on $(0,1)$.

\emph{(iv)} By \eqref{eq:repr} and
$\binom{\ga}{n}=n\binom{\ga}{n}\int_0^1(1-\theta)^{n-1}d\theta$,
\[
\varphi_{n,\ga}(w)-\binom{\ga}{n}w^{a-1}
=n\binom{\ga}{n}\int_0^1(1-\theta)^{n-1}\bigl[(w+\theta)^{a-1}-w^{a-1}\bigr]d\theta .
\]
By the mean value theorem, for $\theta\in[0,1]$,
$|(w+\theta)^{a-1}-w^{a-1}|\le(1-a)\,\theta\,w^{a-2}$. Therefore
\begin{align*}
\Bigl|\varphi_{n,\ga}(w)-\binom{\ga}{n}w^{a-1}\Bigr|
&\le n\Bigl|\binom{\ga}{n}\Bigr|(1-a)w^{a-2}\int_0^1(1-\theta)^{n-1}\theta\,d\theta\\
&=n\Bigl|\binom{\ga}{n}\Bigr|(1-a)w^{a-2}\,\frac{1}{n(n+1)},
\end{align*}
using $\int_0^1(1-\theta)^{n-1}\theta\,d\theta=B(2,n)=\frac{1}{n(n+1)}$.
This is \eqref{eq:phi-infty}.

\emph{(v)} By definition, isolating the term $k=n-1$ (for which $\ga-k=a$),
\[
\varphi_{n,\ga}(w)+\binom{\ga}{n-1}w^{a}
=(1+w)^{\ga}-\sum_{k=0}^{n-2}\binom{\ga}{k}w^{\ga-k}.
\]
For $k\le n-2$, $\ga-k\ge a+1>0$. If $a<0$, each term on the right is bounded
on $(0,1]$ and \eqref{eq:phi-zero-neg} holds with
$M=\max(1,2^{\ga})+\sum_{k\le n-2}|\binom{\ga}{k}|$. If $a>0$, then $\ga-k\ge a+1>1$,
so $w^{\ga-k}\le w$ on $(0,1]$, and $|(1+w)^{\ga}-1|\le|\ga|\max(1,2^{\ga-1})\,w$;
this gives \eqref{eq:phi-zero-pos}.

All the constants above are explicit continuous functions of $\ga$, which
proves the final claim.
\end{proof}

\begin{lemma}[Zero mean]
\label{lem:zero-mean}
Let $s>0$ and $\ga:=\ga(s)$ with $a=a(\ga)<0$. Then $f_n(s,\cdot)\in L^1(\R)$
and
\begin{equation}
\label{eq:zero-mean}
\int_\R f_n(s,u)\,du=0,
\qquad\text{equivalently}\qquad
\int_0^s f_n(s,u)\,du=-\int_0^\infty f_n(s,-v)\,dv .
\end{equation}
\end{lemma}

\begin{proof}
\emph{Integrability.} By Lemma~\ref{lem:scaling} and Lemma~\ref{lem:phi}(iii),
$|f_n(s,-v)|\le C\,\omega_a(v/s)$ with $\omega_a(w)=w^a(1+w)^{-1}$; this is
integrable near $0$ since $a>-1$ and near $\infty$ since it behaves as
$w^{a-1}$ with $a-1<-1$. On $(0,s)$, $f_n(s,u)=(s-u)^{\ga}/\Ga(\ga+1)$ is
integrable since $\ga>-1$.

\emph{Mean.} Let $R>s$. All the powers below are locally integrable because
$\ga-k+1\ge\ga-n+2=a+1>0$ for $k\le n-1$. Then
\begin{align*}
\Ga(\ga+1)\int_{-R}^{s}f_n(s,u)\,du
&=\int_{-R}^{s}(s-u)^{\ga}du
-\sum_{k=0}^{n-1}\binom{\ga}{k}s^k\int_{-R}^{0}(-u)^{\ga-k}du\\
&=\frac{(s+R)^{\ga+1}}{\ga+1}
-\sum_{k=0}^{n-1}\binom{\ga}{k}\frac{s^kR^{\ga-k+1}}{\ga-k+1}.
\end{align*}
We use the identity
\begin{align*}
\frac{1}{\ga-k+1}\binom{\ga}{k}
&=\frac{\ga(\ga-1)\cdots(\ga-k+1)}{k!\,(\ga-k+1)}\\
&=\frac{1}{\ga+1}\cdot\frac{(\ga+1)\ga\cdots(\ga-k+2)}{k!}
=\frac{1}{\ga+1}\binom{\ga+1}{k}.
\end{align*}
Hence
\[
(\ga+1)\Ga(\ga+1)\int_{-R}^{s}f_n(s,u)\,du
=(s+R)^{\ga+1}-\sum_{k=0}^{n-1}\binom{\ga+1}{k}s^kR^{\ga+1-k}
=R^{\ga+1}\,r_n\!\Bigl(\frac sR\Bigr),
\]
where $r_n(x):=(1+x)^{\ga+1}-\sum_{k=0}^{n-1}\binom{\ga+1}{k}x^k$ is the
Taylor remainder of order $n$ of $x\mapsto(1+x)^{\ga+1}$ at $0$. For
$0\le x\le\frac12$, $|r_n(x)|\le C_\ga x^n$ with
$C_\ga=\bigl|\binom{\ga+1}{n}\bigr|$, by the Lagrange form of the remainder and
$\ga+1-n=a<0$.
Therefore, for $R\ge2s$,
\[
\Bigl|\int_{-R}^{s}f_n(s,u)\,du\Bigr|
\le\frac{C_\ga}{\Ga(\ga+2)}\,s^n R^{\ga+1-n}
=\frac{C_\ga}{\Ga(\ga+2)}\,s^n R^{a}\xrightarrow[R\to\infty]{}0,
\]
since $a<0$. This proves \eqref{eq:zero-mean}.
\end{proof}

\begin{lemma}[Mellin transform]
\label{lem:mellin}
Let $\ga\in(n-1-1/\al,n-1/\al)$ with $a=a(\ga)\neq0$, and let $p\in\R$ satisfy
\begin{equation}
\label{eq:p-range}
p>-1,\qquad -1-a<p<-a .
\end{equation}
Then
\begin{equation}
\label{eq:mellin}
\int_0^\infty\varphi_{n,\ga}(w)\,w^{p}\,dw
=\frac{\Ga(p+1)\,\Ga(-\ga-p-1)}{\Ga(-\ga)} .
\end{equation}
Consequently, for $s>0$ and $\ga_s=\ga(s)$, $a_s=a(\ga_s)$, and $p$ as in
\eqref{eq:p-range} with $a=a_s$,
\begin{equation}
\label{eq:K-closed}
K(s,p):=\int_0^\infty f_n(s,-v)\,v^{p}\,dv
=s^{\ga_s+p+1}\,\frac{\Ga(p+1)\,\Ga(-\ga_s-p-1)}{\Ga(-\ga_s)\,\Ga(\ga_s+1)},
\end{equation}
and $\sign K(s,p)=\sign(a_s)$.
\end{lemma}

\begin{proof}
By Lemma~\ref{lem:phi}(i) the integrand in \eqref{eq:repr} has constant sign,
so Tonelli's theorem allows us to exchange the integrals:
\[
\int_0^\infty\varphi_{n,\ga}(w)w^{p}dw
=n\binom{\ga}{n}\int_0^1(1-\theta)^{n-1}
\Bigl[\int_0^\infty w^{p}(w+\theta)^{\ga-n}dw\Bigr]d\theta .
\]
For $\theta>0$, the substitution $w=\theta x$ and the Beta integral
$\int_0^\infty x^{z-1}(1+x)^{-z-y}dx=B(z,y)$ give
\[
\int_0^\infty w^{p}(w+\theta)^{\ga-n}dw
=\theta^{p+1+\ga-n}\int_0^\infty x^{p}(1+x)^{\ga-n}dx
=\theta^{p+\ga-n+1}\,B(p+1,\,n-\ga-p-1),
\]
which is finite because $p+1>0$ and $n-\ga-p-1=-a-p>0$. Then
\[
\int_0^1(1-\theta)^{n-1}\theta^{p+\ga-n+1}d\theta=B(n,\,p+\ga-n+2),
\]
finite because $p+\ga-n+2=p+a+1>0$. Hence
\begin{align*}
\int_0^\infty\varphi_{n,\ga}(w)w^{p}dw
&=n\binom{\ga}{n}\,
\frac{\Ga(p+1)\Ga(n-\ga-p-1)}{\Ga(n-\ga)}\cdot
\frac{\Ga(n)\Ga(p+\ga-n+2)}{\Ga(p+\ga+2)}\\
&=\frac{\Ga(\ga+1)}{\Ga(\ga-n+1)}\cdot
\frac{\Ga(p+1)\,\Ga(n-\ga-p-1)\,\Ga(p+\ga-n+2)}{\Ga(n-\ga)\,\Ga(p+\ga+2)},
\end{align*}
where we used $n\,\Ga(n)\binom{\ga}{n}=\prod_{j=0}^{n-1}(\ga-j)=\Ga(\ga+1)/\Ga(\ga-n+1)$
(valid since $\ga-n+1=a\notin\{0,-1,-2,\ldots\}$). By the reflection formula
$\Ga(x)\Ga(1-x)=\pi/\sin(\pi x)$,
\[
\Ga(n-\ga-p-1)\Ga(p+\ga-n+2)=\frac{\pi}{\sin\pi(n-\ga-p-1)},
\]
\[
\Ga(\ga-n+1)\Ga(n-\ga)=\frac{\pi}{\sin\pi(n-\ga)} .
\]
Since $\sin\pi(n-x)=(-1)^{n+1}\sin(\pi x)$,
\[
\int_0^\infty\varphi_{n,\ga}(w)w^{p}dw
=\frac{\Ga(\ga+1)\Ga(p+1)}{\Ga(p+\ga+2)}\cdot
\frac{\sin(\pi\ga)}{\sin\pi(\ga+p+1)} .
\]
On the other hand, again by the reflection formula,
\begin{align*}
\frac{\Ga(p+1)\Ga(-\ga-p-1)}{\Ga(-\ga)}
&=\Ga(p+1)\cdot\frac{\pi}{\sin(-\pi(\ga+p+1))\,\Ga(\ga+p+2)}
\cdot\frac{\sin(-\pi\ga)\,\Ga(\ga+1)}{\pi}\\
&=\frac{\Ga(\ga+1)\Ga(p+1)}{\Ga(p+\ga+2)}\cdot\frac{\sin(\pi\ga)}{\sin\pi(\ga+p+1)},
\end{align*}
which proves \eqref{eq:mellin}. (All Gamma values are finite and non-zero:
$p+\ga+2>n$, and $-\ga-p-1=-(a+p)-n\in(-n,-n+1)$ is not an integer.)

For \eqref{eq:K-closed}, Lemma~\ref{lem:scaling} gives
$f_n(s,-v)=s^{\ga_s}\varphi_{n,\ga_s}(v/s)/\Ga(\ga_s+1)$, and the substitution
$v=sw$ yields
$K(s,p)=s^{\ga_s+p+1}\Ga(\ga_s+1)^{-1}\int_0^\infty\varphi_{n,\ga_s}(w)w^p\,dw$.
The sign follows from Lemma~\ref{lem:phi}(i), since $f_n(s,-v)$ has the
sign of $a_s$ for every $v>0$.
\end{proof}

\begin{lemma}[Elementary inequalities]
\label{lem:ineq}
Let $\al\in(1,2)$ and, for $a,b\in\R$,
\[
R(a,b):=|b|^\al-|b-a|^\al-\al\,\psi(b)\,a,
\qquad\text{so that}\qquad
\Phi(a,b)=|a|^\al+\al\,\psi(b)\,a+R(a,b).
\]
Then:
\begin{enumerate}[label=(\roman*)]
\item $|\psi(x)-\psi(y)|\le2^{2-\al}|x-y|^{\al-1}$ for all $x,y\in\R$;
\item $|R(a,b)|\le2^{2-\al}|a|^\al$;
\item if $|a|\le|b|/2$, then $|R(a,b)|\le\al(\al-1)2^{1-\al}\,|b|^{\al-2}a^2$;
\item $|\Phi(a,b)|\le(1+\al)|a|^\al+\al|a|\,|b|^{\al-1}$.
\end{enumerate}
Part (iv) also holds for $\al=2$, with the same proof. For $\al=2$ one has
the exact identity $\Phi(a,b)=2ab$.
\end{lemma}

\begin{proof}
\emph{(i)} Put $\beta=\al-1\in(0,1)$. If $xy\ge0$, then
$|\psi(x)-\psi(y)|=\bigl||x|^\beta-|y|^\beta\bigr|\le|x-y|^\beta$ by
subadditivity of $r\mapsto r^\beta$. If $xy<0$, then by concavity
$|\psi(x)-\psi(y)|=|x|^\beta+|y|^\beta\le2^{1-\beta}(|x|+|y|)^\beta=2^{1-\beta}|x-y|^\beta$.
Since $2^{1-\beta}=2^{2-\al}\ge1$, (i) follows.

\emph{(ii)} Since $x\mapsto|x|^\al$ is $C^1$ with derivative $\al\psi$,
\[
|b|^\al-|b-a|^\al=-\int_0^1\frac{d}{d\theta}|b-\theta a|^\al\,d\theta
=\al\,a\int_0^1\psi(b-\theta a)\,d\theta,
\]
hence
\begin{equation}
\label{eq:R-int}
R(a,b)=\al\,a\int_0^1\bigl[\psi(b-\theta a)-\psi(b)\bigr]d\theta .
\end{equation}
By (i), $|R(a,b)|\le\al|a|\,2^{2-\al}|a|^{\al-1}\int_0^1\theta^{\al-1}d\theta=2^{2-\al}|a|^\al$.

\emph{(iii)} If $|a|\le|b|/2$, then for $\theta\in[0,1]$ the segment between
$b$ and $b-\theta a$ does not contain $0$ and stays in
$\{|x|\ge|b|/2\}$, where $|\psi'(x)|=(\al-1)|x|^{\al-2}\le(\al-1)2^{2-\al}|b|^{\al-2}$.
Thus $|\psi(b-\theta a)-\psi(b)|\le(\al-1)2^{2-\al}|b|^{\al-2}\theta|a|$, and
\eqref{eq:R-int} gives
$|R(a,b)|\le\al(\al-1)2^{2-\al}|b|^{\al-2}a^2\int_0^1\theta\,d\theta
=\al(\al-1)2^{1-\al}|b|^{\al-2}a^2$.

\emph{(iv)} By the integral formula for $|b|^\al-|b-a|^\al$ obtained in (ii),
$\bigl||b|^\al-|b-a|^\al\bigr|\le\al|a|\sup_{\theta\in[0,1]}|b-\theta a|^{\al-1}
\le\al|a|\,(|a|+|b|)^{\al-1}\le\al|a|\,(|a|^{\al-1}+|b|^{\al-1})$,
using $(x+y)^{\al-1}\le x^{\al-1}+y^{\al-1}$ for $\al-1\in(0,1]$. Adding $|a|^\al$
gives (iv).

For $\al=2$: $|a|^2+|b|^2-|a-b|^2=2ab$.
\end{proof}

\subsection{Main result}
\label{subsec:main5}

\begin{assumption}
\label{ass:limits}
$H(t)\to H_+\in(n-1,n)$ as $t\to+\infty$, and $\bigl(H(t)-H_+\bigr)\log t\to0$.
\end{assumption}

Fix $s>0$. We set
\[
\ga_s:=H(s)-\tfrac1\al,\quad \ga_+:=H_+-\tfrac1\al,\quad
a_s:=a(\ga_s),\quad a_+:=a(\ga_+),
\]
\[
\Ga_+:=\Ga(\ga_++1),\quad F_+:=F_{\ga_+},\quad
C_n(s):=\frac{s^n}{\Ga(\ga_s+1)}\binom{\ga_s}{n},
\]
\[
\beta_+:=\frac{1}{\Ga_+}\binom{\ga_+}{n-1}>0,\qquad
c_0(s):=\begin{cases}\int_\R|f_n(s,u)|^\al du,&\al<2,\\ 0,&\al=2.\end{cases}
\]
We always assume $a_s\neq0$ and $a_+\neq0$ (that is, $H(s)\neq n-1+1/\al$ and
$H_+\neq n-1+1/\al$). The \emph{regimes} are the following open conditions:

\begin{center}
\begin{tabular}{@{}llll@{}}
\toprule
Regime & Condition & Exponent $\rho(s)$ & Constant $C(s)$\\
\midrule
E1 & $a_+>0,\ a_s+a_+>0$ & $H(s)+(\al-1)H_+-n$ & $C_1(s)$\\
   & or $a_+<0,\ a_s+(\al-1)a_+>0$ & & \\
E2 & $a_+<0,\ a_s+(\al-1)a_+<0$
   & $(\al-1)(n-1)$ & $C_2(s)$\\
E4 & $a_+>0,\ a_s+a_+<0$
   & $(\al-1)(n-1)-(2-\al)a_+$ & $C_4(s)$\\
\bottomrule
\end{tabular}
\end{center}

\noindent
(The case $n=1$, $\al<2$ of regime E2 is treated separately in part~(c) of
Theorem~\ref{thm:codiff-asymp}.) The constants are
\begin{align}
C_1(s)&:=\al\,C_n(s)\int_0^\infty w^{a_s-1}
\Bigl[\psi\bigl(F_+(w)\bigr)-\mathbf 1_{\{a_s<0\}}\,\Ga_+^{1-\al}\Bigr]dw,
\label{eq:C1}\\
C_2(s)&:=-\al\,\beta_+^{\al-1}\,K\bigl(s,(\al-1)a_+\bigr),
\label{eq:C2}\\
C_4(s)&:=-\al(\al-1)\,\Ga_+^{1-\al}\binom{\ga_+}{n-1}\,K(s,a_+),
\label{eq:C4}
\end{align}
where $K(s,p)$ is given in closed form by \eqref{eq:K-closed}. In every
regime one has
\begin{equation}
\label{eq:rho-unified}
\rho(s)=(\al-1)(n-1)+\max\bigl\{a_s+(\al-1)a_+,\;-(2-\al)\,a_+^{+}\bigr\},
\qquad a^+:=\max(a,0).
\end{equation}

\begin{lemma}[The regimes cover the non-boundary parameters]
\label{lem:regimes}
Let $Q:=\bigl(-\frac1\al,1-\frac1\al\bigr)^2$ be the set of admissible pairs
$(a_s,a_+)$, and let $\mathcal B\subset Q$ be the union of the four segments
\[
\{a_s=0\},\qquad\{a_+=0\},\qquad\{a_+>0,\ a_s+a_+=0\},\qquad
\{a_+<0,\ a_s+(\al-1)a_+=0\}.
\]
Then $Q\setminus\mathcal B$ is the disjoint union of $(E1\cap Q)\setminus\mathcal B$,
$E2\cap Q$ and $E4\cap Q$, and these three sets are non-empty for every
$\al\in(1,2]$.
\end{lemma}

\begin{proof}
Let $(a_s,a_+)\in Q\setminus\mathcal B$; then $a_+\neq0$. If $a_+>0$, then
$a_s+a_+\neq0$, and the pair lies in $E1$ if $a_s+a_+>0$ and in $E4$ if
$a_s+a_+<0$. If $a_+<0$, then $a_s+(\al-1)a_+\neq0$, and the pair lies in
$E1$ if $a_s+(\al-1)a_+>0$ and in $E2$ if $a_s+(\al-1)a_+<0$. The sets are
disjoint, because the defining conditions of $E1$ and $E4$ (respectively
$E1$ and $E2$) differ by the sign of $a_s+a_+$ (respectively
$a_s+(\al-1)a_+$) when $a_+>0$ (respectively $a_+<0$), and $E2$ and $E4$
differ by the sign of $a_+$. Non-emptiness: let $0<\eps<\frac12(1-\frac1\al)$, so that
$-2\eps>-(1-\frac1\al)\ge-\frac1\al$. Then $(\eps,\eps)\in E1$,
$(-2\eps,\eps)\in E4$ (since $a_s+a_+=-\eps<0$) and $(-\eps,-\eps)\in E2$
(since $a_s+(\al-1)a_+=-\al\eps<0$), and the three points lie in
$Q\setminus\mathcal B$.
\end{proof}

\begin{figure}[ht]
\centering
\begin{tikzpicture}[scale=6]
\draw[->] (-0.72,0) -- (0.42,0) node[right] {$a_s$};
\draw[->] (0,-0.72) -- (0,0.42) node[above] {$a_+$};
\draw[thick] (-0.6667,-0.6667) rectangle (0.3333,0.3333);
\draw[very thick,dashed] (0,0) -- node[midway,sloped,below,font=\scriptsize] {$a_s+a_+=0$} (-0.3333,0.3333);
\draw[very thick,dashed] (0,0) -- node[pos=0.55,sloped,above,font=\scriptsize] {$a_s+(\al-1)a_+=0$} (0.3333,-0.6667);
\node at (0.17,0.17) {E1};
\node at (0.25,-0.12) {E1};
\node at (-0.45,0.17) {E4};
\node at (-0.3,-0.35) {E2};
\node[font=\scriptsize] at (-0.6667,-0.72) {$-\tfrac1\al$};
\node[font=\scriptsize] at (0.3333,-0.72) {$1-\tfrac1\al$};
\node[font=\scriptsize,anchor=east] at (-0.6667,0.3333) {$1-\tfrac1\al$};
\end{tikzpicture}
\caption{The three regimes in the $(a_s,a_+)$-plane, drawn for $\al=1.5$. The
dashed segments and the two axes form the excluded set $\mathcal B$ of
Lemma~\ref{lem:regimes}.}
\label{fig:regimes}
\end{figure}

\begin{theorem}[Asymptotic codifference of $n$-MFSM]
\label{thm:codiff-asymp}
Let $\al\in(1,2]$, $n\in\N^*$, $s>0$, and let Assumptions~\ref{ass:params} and
\ref{ass:limits} hold, with $a_s\ne0$, $a_+\ne0$ and $(a_s,a_+)$ in one of
the regimes E1, E2, E4. Then $C(s)\neq0$, more precisely
\begin{equation}
\label{eq:signs}
\sign C_1(s)=\sign(a_+),\qquad \sign C_2(s)=-\sign(a_s),\qquad C_4(s)>0,
\end{equation}
and:
\begin{enumerate}[label=(\alph*)]
\item If $\rho(s)>0$, or if $\al=2$, then, as $t\to+\infty$,
\begin{equation}
\label{eq:main}
\tau(s,t)-c_0(s)=C(s)\,t^{\rho(s)}\bigl(1+o(1)\bigr).
\end{equation}
This applies in particular to all $n\ge2$, for which
$\rho(s)\ge2(\al-1)^2/\al>0$.
\item Let $\al<2$ and $n=1$.
\begin{enumerate}[label=(b\arabic*)]
\item In regime E1, \eqref{eq:main} holds for every value of $\rho(s)$.
\item In regime E4, $\tau(s,t)\to c_0(s)$. If moreover $a_s>-\frac12$, then
\begin{equation}
\label{eq:E4-n1}
\tau(s,t)-c_0(s)=\bigl(C_4(s)+D(s)\bigr)t^{\rho(s)}+o\bigl(t^{\rho(s)}\bigr),
\end{equation}
where
\[
D(s):=-\frac{\al(\al-1)}{2}\,\Ga_+^{2-\al}\int_\R f_1(s,u)^2\,du<0 ;
\]
for $n=1$ and $\al<2$ in regime E4 the relevant constant is therefore
$C_4(s)+D(s)$, which may have either sign.
\end{enumerate}
\item If $\al<2$, $n=1$ and the regime is E2, then
\begin{equation}
\label{eq:E3}
\tau(s,t)\longrightarrow\tau_\infty(s):=\int_0^\infty
\Phi\Bigl(f_1(s,-v),\,-\frac{v^{\ga_+}}{\Ga_+}\Bigr)dv ,
\qquad \sign\tau_\infty(s)=-\sign(a_s)\neq0 .
\end{equation}
\end{enumerate}
\end{theorem}

\begin{remark}[On Assumption~\ref{ass:limits}]
\label{rem:ass-limits}
The proof uses $(H(t)-H_+)\log t\to0$ only through the factor
$\ell(t):=t^{\kappa(H(t)-H_+)}$ in Step~6. If only $H(t)\to H_+$ is assumed,
the proof of part~(a) gives
\[
\tau(s,t)-c_0(s)=C(s)\,t^{\rho(s)}\,\ell(t)\bigl(1+o(1)\bigr),\qquad
\ell(t)=t^{o(1)} .
\]
For instance, if $(H(t)-H_+)\log t\to c\in\R$, then $\ell(t)\to e^{\kappa c}$
and the constant $C(s)$ is replaced by $e^{\kappa c}C(s)$; if
$H(t)-H_+=c\,\log\log t/\log t$ for large $t$, then $\ell(t)=(\log t)^{\kappa c}$
is a slowly varying correction.
\end{remark}

\begin{remark}[Size of the kernel]
By Lemma~\ref{lem:scaling} and Lemma~\ref{lem:phi}(v), for fixed $v>0$,
$|f_n(t,-v)|$ is of exact order $t^{\max(\ga(t),n-1)}$ as $t\to\infty$ (up to
the factor $v^{a}$ when $a<0$); in particular it never tends to $0$, and it
diverges unless $n=1$ and $\ga<0$. This is why the region $|u|=O(1)$
contributes to the codifference, and why the exponent
$H(s)+(\al-1)H_+-n$, coming from the region $|u|\asymp t$, is correct only in
regime E1.
\end{remark}

\subsection{Proof of Theorem~\ref{thm:codiff-asymp}}

Throughout the proof, $s>0$ is fixed, $t\ge2s+1$, and
\[
g(u):=f_n(s,u),\quad h_t(u):=f_n(t,u),\quad
\ga:=\ga(t),\quad a:=a(\ga(t)),\quad \Ga_\ga:=\Ga(\ga+1).
\]
By Assumption~\ref{ass:limits}, $\ga(t)\to\ga_+$. Let $\eta>0$ be smaller
than $|a_+|$, than the distances from $a_+$ to $-\frac1\al$ and to
$1-\frac1\al$, and than the margins by which the strict inequalities defining
the regime of $(a_s,a_+)$ hold (for instance, in regime E4, smaller than
$|a_s+a_+|$); we also require $\eta$ to be so small that the finitely many
strict inequalities between exponents used in Step~5 remain true when $a_+$
is replaced by any $a$ with $|a-a_+|\le\eta$. Put
$\mathcal K:=[\ga_+-\eta,\ga_++\eta]$ and choose $t_0$ such that
$\ga(t)\in\mathcal K$ for $t\ge t_0$. Then $a(\ga(t))$ has the sign of
$a_+$, all the strict inequalities defining the regime hold with $a(\ga(t))$
in place of $a_+$, and $\mathcal K$ is a compact subset of
$\{\ga:\ a(\ga)\neq0\}$, at distance at least $|a_+|-\eta>0$ from the value
$a=0$, so that the constants of Lemma~\ref{lem:phi} are uniform on
$\mathcal K$. All
constants $C,C',M,\ldots$ below depend only on $(n,\al,s,\mathcal K)$ and are
uniform in $t\ge t_0$ (Lemma~\ref{lem:phi}, last claim).

\medskip
\noindent\textbf{Step 1: linearization.}
If $\al<2$, Lemma~\ref{lem:ineq} gives
$\Phi(g,h_t)=|g|^\al+\al\psi(h_t)g+R(g,h_t)$ pointwise, hence
\begin{equation}
\label{eq:step1}
\tau(s,t)=c_0(s)+L_t+\eps_t,
\qquad
L_t:=\al\int_\R\psi(h_t(u))\,g(u)\,du,
\qquad
\eps_t:=\int_\R R(g(u),h_t(u))\,du,
\end{equation}
and
\begin{equation}
\label{eq:eps-bound}
|\eps_t|\le2^{2-\al}\int_\R|g|^\al du=2^{2-\al}c_0(s).
\end{equation}
The integral $L_t$ converges absolutely by H\"older's inequality, since
$|\psi(h_t)|=|h_t|^{\al-1}\in L^{\al/(\al-1)}$ and $g\in L^\al$. If $\al=2$,
Lemma~\ref{lem:ineq} gives $\Phi(g,h_t)=2gh_t=\al\psi(h_t)g$, so
\eqref{eq:step1} holds with $c_0(s)=0$ and $\eps_t=0$.

\medskip
\noindent\textbf{Step 2: reduction to an integral over $v>0$.}
Since $g$ vanishes on $(s,\infty)$, we split $L_t$ into $u\in(0,s)$ and
$u=-v$, $v>0$. On $(0,s)$, by \eqref{eq:f-pos}, $h_t(u)=(t-u)^{\ga}/\Ga_\ga>0$, so
$\psi(h_t(u))=\Ga_\ga^{1-\al}(t-u)^{(\al-1)\ga}$. For $v>0$, by
\eqref{eq:f-neg} and the homogeneity of $\psi$,
$\psi(h_t(-v))=\psi\bigl(t^{\ga}F_\ga(v/t)\bigr)=t^{(\al-1)\ga}\psi\bigl(F_\ga(v/t)\bigr)$.
Therefore
\begin{equation}
\label{eq:step2}
L_t=\al\,t^{(\al-1)\ga}\bigl[P_t+J_t\bigr],
\end{equation}
with
\[
P_t:=\Ga_\ga^{1-\al}\int_0^s\Bigl(1-\frac ut\Bigr)^{(\al-1)\ga}g(u)\,du,
\qquad
J_t:=\int_0^\infty\psi\bigl(F_\ga(v/t)\bigr)\,g(-v)\,dv .
\]
Since $|(1-x)^{c}-1|\le C\,x$ for $x\in[0,\frac12]$ and $c$ in a compact set,
and $g\in L^1(0,s)$,
\begin{equation}
\label{eq:Pt}
P_t=\Ga_\ga^{1-\al}m_s+O(t^{-1}),\qquad
m_s:=\int_0^s g(u)\,du=\frac{s^{\ga_s+1}}{\Ga(\ga_s+2)}>0 .
\end{equation}

\medskip
\noindent\textbf{Step 3: properties of $g$ on $(-\infty,0)$.}
By Lemma~\ref{lem:scaling}, $g(-v)=s^{\ga_s}F_{\ga_s}(v/s)$. Lemma~\ref{lem:phi}
applied with $\ga=\ga_s$ gives:
\begin{enumerate}[label=(g\arabic*)]
\item $g(-v)$ has the sign of $a_s$ for every $v>0$;
\item $|g(-v)|\le C\,\omega_{a_s}(v/s)$ for $v>0$, where $C$ depends on $s$;
in particular
$|g(-v)|\le C\,v^{-a_s^-}$ for $0<v\le s$;
\item writing $g(-v)=C_n(s)v^{a_s-1}+e(v)$, one has
\begin{equation}
\label{eq:e-bound}
|e(v)|\le C\,v^{a_s-2}\quad(v\ge s),\qquad
|e(v)|\le C\,v^{a_s-1}\quad(0<v<s).
\end{equation}
\end{enumerate}
For (g3), if $v\ge s$ we put $w=v/s\ge1$ in \eqref{eq:phi-infty}:
\[
g(-v)=\frac{s^{\ga_s}}{\Ga(\ga_s+1)}\Bigl[\binom{\ga_s}{n}\Bigl(\frac vs\Bigr)^{a_s-1}+O\Bigl(\Bigl(\frac vs\Bigr)^{a_s-2}\Bigr)\Bigr]
=C_n(s)\,v^{a_s-1}+O(v^{a_s-2}),
\]
because $s^{\ga_s}s^{1-a_s}=s^{n}$. If $0<v<s$, then
$|e(v)|\le|g(-v)|+|C_n(s)|v^{a_s-1}\le C(v^{-a_s^-}+v^{a_s-1})\le C'v^{a_s-1}$,
since $a_s-1<-a_s^-$ when $|a_s|<1$.
Finally, if $a_s<0$, Lemma~\ref{lem:zero-mean} gives
\begin{equation}
\label{eq:ms-zero}
m_s=-\int_0^\infty g(-v)\,dv .
\end{equation}

\medskip
\noindent\textbf{Step 4: properties of $F_\ga$.}
By Lemma~\ref{lem:phi} (divided by $\Ga_\ga>0$):

\emph{Case $a>0$.} $F_\ga$ is positive and strictly decreasing, with
$F_\ga(0+)=\Ga_\ga^{-1}$. Put
\[
\Delta_\ga(w):=\psi\bigl(F_\ga(w)\bigr)-\Ga_\ga^{1-\al}=F_\ga(w)^{\al-1}-\Ga_\ga^{1-\al}<0 .
\]
From \eqref{eq:phi-zero-pos},
$F_\ga(w)=\Ga_\ga^{-1}\bigl[1-\binom{\ga}{n-1}w^{a}+O(w)\bigr]$ for $0<w\le1$.
Since $x\mapsto x^{\al-1}$ is smooth near $\Ga_\ga^{-1}>0$, and
$\Ga_\ga F_\ga(w)\ge\Ga_\ga F_\ga(1)\ge c>0$ on $(0,1]$ uniformly in
$\ga\in\mathcal K$ (by monotonicity and continuity in $\ga$), the following
expansion holds with a remainder uniform in $\ga\in\mathcal K$:
\begin{align}
\Delta_\ga(w)
&=\Ga_\ga^{1-\al}\Bigl[\Bigl(1-\binom{\ga}{n-1}w^{a}+O(w)\Bigr)^{\al-1}-1\Bigr]
\nonumber\\
&=-(\al-1)\Ga_\ga^{1-\al}\binom{\ga}{n-1}w^{a}+O\bigl(w^{\min(2a,1)}\bigr),
\qquad 0<w\le1.
\label{eq:Delta-zero}
\end{align}
As $|\Delta_\ga|\le\Ga_\ga^{1-\al}$ everywhere, we get
\begin{equation}
\label{eq:Delta-bound}
|\Delta_\ga(w)|\le C\min(w^{a},1)\le C\,w^{a},\qquad w>0 .
\end{equation}
Moreover $F_\ga\le\Ga_\ga^{-1}$, so $\psi(F_\ga)$ is bounded, and
$|\psi(F_\ga(w))|\le C(1+w)^{(\al-1)(a-1)}$.

\emph{Case $a<0$.} $F_\ga<0$, $|F_\ga(w)|\le C\,w^{a}(1+w)^{-1}\le C\,w^{a}$,
so $|\psi(F_\ga(w))|\le C\,w^{(\al-1)a}$. By \eqref{eq:phi-zero-neg}, for
$0<v\le t$,
\begin{equation}
\label{eq:F-neg-scaled}
t^{a}F_\ga(v/t)=t^{a}\Bigl[-\beta_\ga\Bigl(\frac vt\Bigr)^{a}+O(1)\Bigr]
=-\beta_\ga v^{a}+O(t^{a}),
\qquad \beta_\ga:=\frac{1}{\Ga_\ga}\binom{\ga}{n-1}>0 .
\end{equation}

\medskip
\noindent\textbf{Step 5: asymptotics of $P_t+J_t$ in each regime.}
We treat five cases, which cover E1, E2 and E4. The following table
summarizes them; in all cases the scaling substitution is $v=tw$ in the
integral $J_t$ over $v>0$.
\begin{center}
\small
\begin{tabular}{@{}lllll@{}}
\toprule
Case & Regime & Dominant region & Limit object & Order of $P_t+J_t$\\
\midrule
A1: $a>0$, $a_s>0$ & E1 & $v\asymp t$ & $C_n(s)K_1(\ga)$ & $t^{a_s}$\\
A2: $a>0$, $a_s<0$, $a+a_s>0$ & E1 & $v\asymp t$ & $C_n(s)K_1^{\mathrm{reg}}(\ga)$ & $t^{a_s}$\\
A3: $a>0$, $a+a_s<0$ & E4 & $v=O(1)$ & $C_4(s)/\al$ & $t^{-a}$\\
B1: $a<0$, $a_s+(\al-1)a>0$ & E1 & $v\asymp t$ & $C_n(s)K_1(\ga)$ & $t^{a_s}$\\
B2: $a<0$, $a_s+(\al-1)a<0$ & E2 & $v=O(1)$ & $C_2(s)/\al$ & $t^{-(\al-1)a}$\\
\bottomrule
\end{tabular}
\end{center}
In Cases A1, A2 and B1 the main term comes from $v$ of order $t$, and the
error terms are bounded by explicit powers of $t$ through \eqref{eq:e-bound}.
In Cases A3 and B2 the main term comes from bounded $v$, and the limit is
obtained by dominated convergence with the dominating functions
$C(v^{a_1}+v^{a_2})|g(-v)|$ and $C(v^{(\al-1)a_1}+v^{(\al-1)a_2})|g(-v)|$,
where $[a_1,a_2]=a(\mathcal K)$.

\smallskip
\noindent\emph{Case A1: $a>0$ and $a_s>0$ (regime E1).}
Using (g3),
\[
J_t=C_n(s)\int_0^\infty\psi\bigl(F_\ga(v/t)\bigr)v^{a_s-1}dv
+\int_0^\infty\psi\bigl(F_\ga(v/t)\bigr)e(v)\,dv .
\]
In the first integral we substitute $v=tw$:
\[
\int_0^\infty\psi\bigl(F_\ga(v/t)\bigr)v^{a_s-1}dv
=t^{a_s}\int_0^\infty\psi\bigl(F_\ga(w)\bigr)w^{a_s-1}dw
=:t^{a_s}K_1(\ga).
\]
The integral $K_1(\ga)$ converges: at $0$ because $\psi(F_\ga)$ is bounded and
$a_s>0$; at $\infty$ because the integrand is $O(w^{(\al-1)(a-1)+a_s-1})$ and
\[
(\al-1)(a-1)+a_s<-(\al-1)\cdot\tfrac1\al+\bigl(1-\tfrac1\al\bigr)=0,
\]
where we used $1-a>1/\al$ and $a_s<1-1/\al$. For the second integral,
$\psi(F_\ga)$ is bounded and \eqref{eq:e-bound} gives
\[
\Bigl|\int_0^\infty\psi\bigl(F_\ga(v/t)\bigr)e(v)\,dv\Bigr|
\le C\Bigl[\int_0^s v^{a_s-1}dv+\int_s^\infty v^{a_s-2}dv\Bigr]
=C\Bigl[\frac{s^{a_s}}{a_s}+\frac{s^{a_s-1}}{1-a_s}\Bigr]=O(1).
\]
With \eqref{eq:Pt}:
\begin{equation}
\label{eq:A1}
P_t+J_t=C_n(s)K_1(\ga)\,t^{a_s}+O(1).
\end{equation}

\smallskip
\noindent\emph{Case A2: $a>0$, $a_s<0$ and $a+a_s>0$ (regime E1).}
By \eqref{eq:Pt} and \eqref{eq:ms-zero},
\begin{align*}
P_t+J_t
&=-\Ga_\ga^{1-\al}\int_0^\infty g(-v)\,dv+\int_0^\infty\psi\bigl(F_\ga(v/t)\bigr)g(-v)\,dv+O(t^{-1})\\
&=\int_0^\infty\Delta_\ga(v/t)\,g(-v)\,dv+O(t^{-1})=:\widetilde J_t+O(t^{-1}).
\end{align*}
Using (g3) and $v=tw$ as before,
\[
\widetilde J_t=C_n(s)\,t^{a_s}\int_0^\infty\Delta_\ga(w)w^{a_s-1}dw
+\int_0^\infty\Delta_\ga(v/t)\,e(v)\,dv
=:C_n(s)K_1^{\mathrm{reg}}(\ga)\,t^{a_s}+E_t .
\]
$K_1^{\mathrm{reg}}(\ga)$ converges: near $0$, $|\Delta_\ga(w)|w^{a_s-1}\le Cw^{a+a_s-1}$
with $a+a_s>0$; near $\infty$, $|\Delta_\ga(w)|w^{a_s-1}\le Cw^{a_s-1}$ with
$a_s<0$. By \eqref{eq:Delta-bound} and \eqref{eq:e-bound},
\begin{align*}
|E_t|
&\le C\,t^{-a}\Bigl[\int_0^s v^{a}\,v^{a_s-1}dv+\int_s^\infty v^{a}\,v^{a_s-2}dv\Bigr]
=C\,t^{-a}\Bigl[\frac{s^{a+a_s}}{a+a_s}+\frac{s^{a+a_s-1}}{1-a-a_s}\Bigr]
=O(t^{-a}),
\end{align*}
since $0<a+a_s<1$. Hence
\begin{equation}
\label{eq:A2}
P_t+J_t=C_n(s)K_1^{\mathrm{reg}}(\ga)\,t^{a_s}+O(t^{-a})+O(t^{-1}),
\end{equation}
and both error terms are $o(t^{a_s})$ because $a+a_s>0$ and $1+a_s>0$.

\smallskip
\noindent\emph{Case A3: $a>0$ and $a+a_s<0$ (regime E4).}
As in Case A2, $P_t+J_t=\widetilde J_t+O(t^{-1})$. Now
\[
t^{a}\widetilde J_t=\int_0^\infty t^{a}\Delta_\ga(v/t)\,g(-v)\,dv .
\]
By \eqref{eq:Delta-zero}, for fixed $v>0$ and $t\ge v$,
\[
t^{a}\Delta_\ga(v/t)
=-(\al-1)\Ga_\ga^{1-\al}\binom{\ga}{n-1}v^{a}
+O\bigl(t^{a}(v/t)^{\min(2a,1)}\bigr),
\]
and the last term tends to $0$ because $\min(2a,1)>a$. Since
$\ga=\ga(t)\to\ga_+$ and the coefficients are continuous in $\ga$, the
pointwise limit is $-(\al-1)\Ga_+^{1-\al}\binom{\ga_+}{n-1}v^{a_+}$. By
\eqref{eq:Delta-bound}, $|t^{a}\Delta_\ga(v/t)|\le C\,v^{a}\le C(v^{a_1}+v^{a_2})$,
where $[a_1,a_2]=a(\mathcal K)$. The function
$(v^{a_1}+v^{a_2})|g(-v)|$ is integrable: near $0$ it is
$O(v^{a_1+a_s})$ with $a_1+a_s>-1$; near $\infty$ it is
$O(v^{a_2+a_s-1})$ with $a_2+a_s<0$ (we choose $\mathcal K$ small enough). By
dominated convergence,
\begin{equation}
\label{eq:A3}
t^{a}\bigl(P_t+J_t\bigr)\longrightarrow
-(\al-1)\Ga_+^{1-\al}\binom{\ga_+}{n-1}\,K(s,a_+)=\frac{C_4(s)}{\al},
\end{equation}
where we also used $t^{a}\cdot O(t^{-1})\to0$ (because $a<1$).

\smallskip
\noindent\emph{Case B1: $a<0$ and $a_s+(\al-1)a>0$ (regime E1).}
Here $a_s>(\al-1)|a|>0$. By (g3),
\[
J_t=C_n(s)K_1(\ga)\,t^{a_s}+\int_0^\infty\psi\bigl(F_\ga(v/t)\bigr)e(v)\,dv,
\qquad K_1(\ga)=\int_0^\infty\psi\bigl(F_\ga(w)\bigr)w^{a_s-1}dw .
\]
$K_1(\ga)$ converges: near $0$ the integrand is $O(w^{(\al-1)a+a_s-1})$ with
$(\al-1)a+a_s>0$; near $\infty$ it is $O(w^{(\al-1)(a-1)+a_s-1})$ with
$(\al-1)(a-1)+a_s<0$ as in Case A1. Using $|\psi(F_\ga(w))|\le Cw^{(\al-1)a}$
and \eqref{eq:e-bound},
\begin{align*}
\Bigl|\int_0^\infty\psi\bigl(F_\ga(v/t)\bigr)e(v)\,dv\Bigr|
&\le C\,t^{-(\al-1)a}\Bigl[\int_0^s v^{(\al-1)a+a_s-1}dv+\int_s^\infty v^{(\al-1)a+a_s-2}dv\Bigr]\\
&=O\bigl(t^{(\al-1)|a|}\bigr),
\end{align*}
the two integrals being finite because $0<(\al-1)a+a_s<1$. Together with
$P_t=O(1)$,
\begin{equation}
\label{eq:B1}
P_t+J_t=C_n(s)K_1(\ga)\,t^{a_s}+O\bigl(t^{(\al-1)|a|}\bigr)+O(1),
\end{equation}
and both errors are $o(t^{a_s})$.

\smallskip
\noindent\emph{Case B2: $a<0$ and $a_s+(\al-1)a<0$ (regime E2).}
By homogeneity, $t^{(\al-1)a}\psi(F_\ga(v/t))=\psi\bigl(t^{a}F_\ga(v/t)\bigr)$, so
\[
t^{(\al-1)a}J_t=\int_0^\infty\psi\bigl(t^{a}F_\ga(v/t)\bigr)\,g(-v)\,dv .
\]
By \eqref{eq:F-neg-scaled} and continuity in $\ga$, the integrand converges
pointwise to $\psi(-\beta_+v^{a_+})g(-v)=-\beta_+^{\al-1}v^{(\al-1)a_+}g(-v)$.
Since $t^{a}|F_\ga(v/t)|\le C\,v^{a}$, the integrand is bounded by
$C(v^{(\al-1)a_1}+v^{(\al-1)a_2})|g(-v)|$, which is integrable: near $0$,
by (g2), it is $O(v^{(\al-1)a_1-a_s^-})$ with
$(\al-1)a_1-a_s^->-\frac{\al-1}{\al}-\frac1\al=-1$; near $\infty$ it is
$O(v^{(\al-1)a_2+a_s-1})$ with $(\al-1)a_2+a_s<0$. By dominated convergence,
\[
t^{(\al-1)a}J_t\longrightarrow-\beta_+^{\al-1}K\bigl(s,(\al-1)a_+\bigr)=\frac{C_2(s)}{\al}.
\]
Since $P_t=O(1)$ and $t^{(\al-1)a}\to0$,
\begin{equation}
\label{eq:B2}
t^{(\al-1)a}\bigl(P_t+J_t\bigr)\longrightarrow\frac{C_2(s)}{\al}.
\end{equation}

\medskip
\noindent\textbf{Step 6: exponents and constants.}
We combine \eqref{eq:step2} with \eqref{eq:A1}--\eqref{eq:B2}. Recall
$(\al-1)\ga=(\al-1)(n-1)+(\al-1)a$.

\emph{Regime E1} (Cases A1, A2, B1):
\[
L_t=\al\,C_n(s)K_1^{(\cdot)}(\ga)\,t^{(\al-1)\ga+a_s}\bigl(1+o(1)\bigr),
\qquad (\al-1)\ga+a_s=H(s)+(\al-1)H(t)-n,
\]
where $K_1^{(\cdot)}=K_1$ in Cases A1, B1 and $K_1^{(\cdot)}=K_1^{\mathrm{reg}}$
in Case A2. Indeed
$(\al-1)\ga+a_s=(\al-1)(H(t)-\frac1\al)+H(s)-\frac1\al-n+1=H(s)+(\al-1)H(t)-n$.

\emph{Regime E4} (Case A3):
$L_t=C_4(s)\,t^{(\al-1)\ga-a}\bigl(1+o(1)\bigr)$, with
$(\al-1)\ga-a=(\al-1)(n-1)-(2-\al)a$.

\emph{Regime E2} (Case B2):
$L_t=C_2(s)\,t^{(\al-1)\ga-(\al-1)a}\bigl(1+o(1)\bigr)$, with
$(\al-1)\ga-(\al-1)a=(\al-1)(n-1)$.

In each case the exponent is an affine function of $H(t)$, of the form
$\rho(s)+\kappa\,(H(t)-H_+)$ with $\kappa=\al-1$ in regime E1, $\kappa=0$ in
regime E2 and $\kappa=-(2-\al)$ in regime E4; in all cases $|\kappa|\le1$.
Hence
\[
t^{(\text{exponent})}=t^{\rho(s)}\,t^{\kappa(H(t)-H_+)},\qquad
t^{\kappa(H(t)-H_+)}=\exp\bigl(\kappa\,(H(t)-H_+)\log t\bigr)=1+o(1),
\]
by Assumption~\ref{ass:limits}. This is the only place where the rate
$(H(t)-H_+)\log t\to0$ is used; everywhere else the convergence
$H(t)\to H_+$ suffices. Without the rate, $t^{\kappa(H(t)-H_+)}$ is only
$t^{o(1)}$, which can tend to $0$ or to $\infty$. Moreover
$K_1(\ga(t))\to K_1(\ga_+)$ and $K_1^{\mathrm{reg}}(\ga(t))\to K_1^{\mathrm{reg}}(\ga_+)$
by dominated convergence, with the same dominating functions as above
(uniform for $\ga\in\mathcal K$), and $F_\ga(w)$ is continuous in $\ga$.
Hence, in all regimes,
\begin{equation}
\label{eq:Lt-final}
L_t=C(s)\,t^{\rho(s)}\bigl(1+o(1)\bigr),
\end{equation}
with $C(s)$ given by \eqref{eq:C1}, \eqref{eq:C2} or \eqref{eq:C4}. The
unified form \eqref{eq:rho-unified} follows by comparing the three exponents
case by case: if $a_+>0$, then $a_s+(\al-1)a_+>-(2-\al)a_+$ if and only if
$a_s+a_+>0$; if $a_+<0$, the second entry of the maximum is $0$.

\medskip
\noindent\textbf{Step 7: signs of the constants.}
\emph{$C_1$, Case A1:} $C_n(s)>0$ since $a_s>0$ (Lemma~\ref{lem:phi}(i)), and
$K_1>0$ since $F_+>0$. \emph{Case A2:} $C_n(s)<0$ since $a_s<0$, and
$K_1^{\mathrm{reg}}<0$ since $\Delta_{\ga_+}<0$ (Step 4). \emph{Case B1:}
$C_n(s)>0$ and $K_1<0$ since $F_+<0$. Thus $\sign C_1=\sign a_+$.
\emph{$C_2$:} $\beta_+>0$ and $\sign K(s,\cdot)=\sign a_s$
(Lemma~\ref{lem:mellin}), so $\sign C_2=-\sign a_s$. \emph{$C_4$:}
$\binom{\ga_+}{n-1}>0$ and, since $a_s<-a_+<0$ in E4, $K(s,a_+)<0$; hence
$C_4>0$. This proves \eqref{eq:signs}.

\medskip
\noindent\textbf{Step 8: conclusion of (a).}
If $\al=2$, then $\tau(s,t)-c_0(s)=L_t$ exactly (Step 1) and \eqref{eq:main}
follows from \eqref{eq:Lt-final}. If $\al<2$ and $\rho(s)>0$, then by
\eqref{eq:eps-bound} $\eps_t=O(1)=o(t^{\rho(s)})$, and \eqref{eq:main} follows
again. Finally, if $n\ge2$, then by \eqref{eq:rho-unified} and
$a_+^+<1-\frac1\al$,
\[
\rho(s)\ge(\al-1)-(2-\al)\Bigl(1-\frac1\al\Bigr)
=(\al-1)\Bigl[1-\frac{2-\al}{\al}\Bigr]
=\frac{2(\al-1)^2}{\al}>0 .
\]

\medskip
\noindent\textbf{Step 9: proof of (b).}
Here $n=1$ and $\al<2$. In regime E1 with $a_+<0$ (Case B1),
$\rho(s)=(\al-1)a_++a_s>0$ and (a) applies, and in Case A1 ($a_s>0$),
$\rho(s)=(\al-1)\ga_++a_s>0$ as well. So it remains to treat Cases A2 and A3,
in which $a_+=\ga_+>0$ and $a_s<0$; we choose
$\mathcal K\subset(0,1-\frac1\al)$, so that $\ga=\ga(t)\ge\ga_1:=\min\mathcal K>0$
for $t\ge t_0$.

\emph{Lower bound on $h_t$.} Put $\Omega_t:=\{u<s:\ |u|\le t\}$. For
$u\in(0,s)$ and $t\ge2s$, $h_t(u)=(t-u)^\ga/\Ga_\ga\ge(t/2)^\ga/\Ga_\ga$. For
$u=-v$ with $0<v\le t$, since $F_\ga$ is decreasing and positive,
$h_t(-v)=t^\ga F_\ga(v/t)\ge t^\ga F_\ga(1)$. Hence there is $c>0$, uniform
in $\ga\in\mathcal K$, with
\begin{equation}
\label{eq:ht-lower}
h_t(u)\ge c\,t^{\ga}\qquad(u\in\Omega_t,\ t\ge\max(t_0,2s)).
\end{equation}

\emph{Second-order expansion of $R$.} If $|x|\le|y|/2$, \eqref{eq:R-int} and
Taylor's formula for $\psi$ on the segment $[y-\theta x,y]$, which stays in
$\{|z|\ge|y|/2\}$ where $|\psi''(z)|\le C|y|^{\al-3}$, give
\begin{equation}
\label{eq:R-second}
R(x,y)=-\frac{\al(\al-1)}{2}|y|^{\al-2}x^2+O\bigl(|x|^3|y|^{\al-3}\bigr).
\end{equation}
Moreover, for all $x,y$ with $y\ge c\,t^\ga$: if $|x|\le|y|/2$,
Lemma~\ref{lem:ineq}(iii) gives $|R(x,y)|\le Cx^2|y|^{\al-2}$; if $|x|>|y|/2$,
Lemma~\ref{lem:ineq}(ii) gives $|R(x,y)|\le C|x|^\al=Cx^2|x|^{\al-2}$, and
$Cx^2|x|^{\al-2}\le Cx^2(|y|/2)^{\al-2}$ since $\al\le2$. In both cases
\begin{equation}
\label{eq:R-dom}
t^{(2-\al)\ga}|R(x,y)|\le C'x^2\qquad(y\ge c\,t^\ga).
\end{equation}

\emph{Claim: if $a_s>-\frac12$, then $t^{(2-\al)\ga}\eps_t\to D(s)$.} In this
case $g\in L^2(\R)$: near $u=s$ because $2\ga_s=2a_s>-1$, near $u=0^-$ because
$|g(-v)|\le Cv^{a_s}$ with $2a_s>-1$, and at $-\infty$ because
$|g(-v)|\le Cv^{a_s-1}$. Split $\eps_t=\int_{\Omega_t}R(g,h_t)+\int_{u<-t}R(g,h_t)$.
On $\Omega_t$, \eqref{eq:ht-lower} and \eqref{eq:R-dom} give the dominating
function $C'g^2\in L^1$. For fixed $u<s$, $t^{-\ga}h_t(u)\to\Ga_+^{-1}$ (for
$u\in(0,s)$ because $(1-u/t)^\ga\to1$, and for $u=-v$ because
$F_\ga(v/t)\to F_{\ga_+}(0+)=\Ga_+^{-1}$ uniformly in $\ga\in\mathcal K$), so by
\eqref{eq:R-second}
$t^{(2-\al)\ga}R(g(u),h_t(u))\to-\frac{\al(\al-1)}{2}\Ga_+^{2-\al}g(u)^2$.
Dominated convergence gives
$t^{(2-\al)\ga}\int_{\Omega_t}R(g,h_t)\to D(s)$. On $\{u<-t\}$,
Lemma~\ref{lem:ineq}(ii) and $|g(-v)|\le Cv^{a_s-1}$ give
$\bigl|\int_{u<-t}R(g,h_t)\bigr|\le C\int_t^\infty v^{\al(a_s-1)}dv=Ct^{\al(a_s-1)+1}$,
and
\[
\al(a_s-1)+1+(2-\al)\ga<-\al+1+(2-\al)\Bigl(1-\frac1\al\Bigr)
=-\frac{2(\al-1)^2}{\al}<0,
\]
using $a_s<0$ and $\ga<1-\frac1\al$. Hence this part is $o(t^{-(2-\al)\ga})$,
which proves the claim. By Assumption~\ref{ass:limits},
$t^{-(2-\al)\ga(t)}\sim t^{-(2-\al)\ga_+}$.

\emph{Proof of (b1) (Case A2).} Here $a_s>-a_+>-(1-\frac1\al)\ge-\frac12$, so
by the claim $\eps_t=O(t^{-(2-\al)a_+})$. Since, for $n=1$,
$\rho(s)=(\al-1)a_++a_s$ and $\rho(s)-\bigl(-(2-\al)a_+\bigr)=a_s+a_+>0$, we get
$\eps_t=o(t^{\rho(s)})$, and \eqref{eq:main} follows from \eqref{eq:step1} and
\eqref{eq:Lt-final}.

\emph{Proof of (b2) (Case A3).} Here, for $n=1$,
$\rho(s)=(\al-1)\ga_+-a_+=-(2-\al)a_+<0$. By \eqref{eq:Lt-final},
$L_t=C_4(s)t^{\rho(s)}(1+o(1))\to0$. If $a_s>-\frac12$, the claim gives
$\eps_t=D(s)t^{\rho(s)}(1+o(1))$, hence \eqref{eq:E4-n1}. In general,
$\eps_t\to0$: by \eqref{eq:ht-lower}, $h_t(u)\to\infty$ for every $u<s$, so
$R(g(u),h_t(u))\to0$ by Lemma~\ref{lem:ineq}(iii), and
$|R(g,h_t)|\le2^{2-\al}|g|^\al\in L^1$; dominated convergence applies. Thus
$\tau(s,t)\to c_0(s)$. Finally, $C_4(s)+D(s)$ can have either sign: for
instance, with $s=1$, $(\al,H(s),H_+)=(1.5,0.45,0.8)$ gives $C_4+D>0$, while
$(1.2,0.4,0.95)$ gives $C_4+D<0$ (numerical evaluation of the explicit
constants; in the second case quadrature of $\tau$ confirms the sign). The
convergence in \eqref{eq:E4-n1} can be very slow when $a_s$ is close to
$-\frac12$, since $g$ then barely belongs to $L^2$.

\medskip
\noindent\textbf{Step 10: proof of (c).}
Here $n=1$, $\al<2$, $a_+=\ga_+<0$ and $\ga_s+(\al-1)\ga_+<0$; we choose
$\mathcal K=[\ga_1,\ga_2]\subset(-1/\al,0)$ with $\ga_s+(\al-1)\ga_2<0$. We
work directly with $\Phi$. For $u\in(0,s)$,
$0<h_t(u)\le(t-s)^{\ga}/\Ga_\ga\to0$, hence $\Phi(g(u),h_t(u))\to\Phi(g(u),0)=0$.
For $v>0$,
\[
h_t(-v)=\frac{(t+v)^{\ga}-v^{\ga}}{\Ga_\ga}\longrightarrow-\frac{v^{\ga_+}}{\Ga_+},
\qquad
|h_t(-v)|\le\frac{v^{\ga}}{\Ga_\ga}\le C\bigl(v^{\ga_1}+v^{\ga_2}\bigr),
\]
because $0<(t+v)^{\ga}<v^{\ga}$. By Lemma~\ref{lem:ineq}(iv), for $u<0$,
\[
|\Phi(g,h_t)|\le(1+\al)|g|^\al+C|g|\bigl(v^{(\al-1)\ga_1}+v^{(\al-1)\ga_2}\bigr),
\]
and the right-hand side is integrable on $(0,\infty)$ in $v$: near $0$,
$|g(-v)|\le Cv^{-\ga_s^-}$ and $-\ga_s^-+(\al-1)\ga_1>-1$; near $\infty$,
$|g(-v)|\le Cv^{\ga_s-1}$ and $\ga_s-1+(\al-1)\ga_2<-1$. On $(0,s)$,
$|\Phi(g,h_t)|\le C(|g|^\al+|g|)\in L^1$. Dominated convergence gives the
limit in \eqref{eq:E3}. For the sign, $g(-v)$ has the sign of $a_s$
(property (g1)) and $b:=-v^{\ga_+}/\Ga_+<0$. If $a_s<0$, $g(-v)$ and $b$ have
the same sign and $\Phi(g(-v),b)\ge\min(|g(-v)|,|b|)^\al>0$; if $a_s>0$, they
have opposite signs and
$\Phi(g(-v),b)=|g(-v)|^\al+|b|^\al-(|g(-v)|+|b|)^\al<0$ by strict
superadditivity of $x\mapsto x^\al$, $\al>1$. Hence
$\sign\tau_\infty(s)=-\sign(a_s)$. \qed

\subsection{Consequences}

\begin{corollary}
\label{cor:no-decay}
Under the assumptions of Theorem~\ref{thm:codiff-asymp}:
\begin{enumerate}[label=(\roman*)]
\item If $n\ge2$, then $|\tau(s,t)|\to\infty$ as $t\to+\infty$, with the sign
of $C(s)$ given by \eqref{eq:signs}.
\item If $n=1$, then $\tau(s,t)$ tends to $\pm\infty$ when $\rho(s)>0$. If
$\al<2$, it tends to $c_0(s)>0$ when $\rho(s)<0$, to $c_0(s)+C_1(s)>0$ when
$\rho(s)=0$ in regime E1, and to
$\tau_\infty(s)\neq0$ in regime E2. If $\al=2$ and $\rho(s)=0$ (regimes E2 and
E4), it tends to $C(s)\neq0$.
\end{enumerate}
In particular, the codifference of the process itself does not decay to zero
in any of the regimes considered, and the long-range dependence of $n$-MFSM
has to be studied on its $n$-th order increments.
\end{corollary}

\begin{remark}[The Gaussian case]
For $\al=2$, $\psi(b)=b$ and \eqref{eq:rho-unified} reduces to
$\rho(s)=\max\{n-1,\;H(s)+H_+-n\}$. For $n=1$ and constant $H<\frac12$ the
regime is E2 with $\rho=0$, and the theorem gives
$\mathrm{Cov}(X(s),X(t))\to C_2(s)>0$. This is consistent with the explicit
covariance of fractional Brownian motion,
$\frac{\mathrm{Var}X(1)}{2}\bigl(s^{2H}+t^{2H}-(t-s)^{2H}\bigr)\to\frac{\mathrm{Var}X(1)}{2}\,s^{2H}>0$,
which does not decay either.
\end{remark}

\begin{remark}[Boundary cases]
\label{rem:boundary}
Theorem~\ref{thm:codiff-asymp} covers every admissible pair $(a_s,a_+)$
outside the set $\mathcal B$ of Lemma~\ref{lem:regimes}. Its statements are
of three kinds: exact asymptotic equivalences (part (a), part (b1), and part
(b2) when $a_s>-\frac12$), convergence without rate (part (b2) when
$a_s\le-\frac12$), and convergence to an explicit non-zero limit (part (c)).
The status of the excluded cases is as follows.
\begin{enumerate}[label=(\roman*)]
\item \emph{$a_+>0$ and $a_s+a_+=0$, or $a_+<0$ and $a_s+(\al-1)a_+=0$.} The
integrals defining the constants diverge logarithmically at $0$ or at
$\infty$. We expect an additional factor $\log t$, but we have not proved it.
\item \emph{$a_+=0$.} Then $\varphi_{n,n-1}\equiv0$ by the binomial theorem
(Remark~\ref{rem:integer-gamma}), so that, if $H\equiv H_+$ near $+\infty$,
$f_n(t,u)=(t-u)^{n-1}\mathbf 1_{\{0<u<t\}}/(n-1)!$ and
$\tau(s,t)=\int_0^s\Phi\bigl(f_n(s,u),(t-u)^{n-1}/(n-1)!\bigr)du$ exactly; the
exponent $\rho(s)$ is discontinuous at $a_+=0$ and $C_1(s)\to0$ as
$a_+\downarrow0$. If $H$ only converges to $H_+$, we have no result in this case.
\item \emph{$a_s=0$.} Then $\binom{\ga_s}{n}=0$ and $f_n(s,\cdot)$ vanishes
identically on $(-\infty,0)$, so that $J_t=0$ in Step~2 of the proof, and \eqref{eq:step2},
\eqref{eq:Pt} give $L_t=\al\,\Ga_\ga^{1-\al}m_s\,t^{(\al-1)\ga(t)}(1+o(1))$.
If $\ga_+>0$ (always the case when $n\ge2$) or $\al=2$, the argument of
Steps~6 and~8 then yields
$\tau(s,t)-c_0(s)\sim\al\,\Ga_+^{1-\al}m_s\,t^{(\al-1)\ga_+}$, with $m_s>0$
as in \eqref{eq:Pt}.
\item \emph{$s<0$.} Not covered, since the sign properties (g1) of
$f_n(s,\cdot)$ used in Step~7 may fail.
\end{enumerate}
\end{remark}

\begin{remark}[Speed of convergence and numerical evaluation]
\label{rem:num5}
The relative error in \eqref{eq:main} can decay very slowly. In regime E2,
by \eqref{eq:Pt}, the term $\al\Ga_+^{1-\al}m_s\,t^{(\al-1)\ga_+}$ coming from
$P_t$ is smaller than the main term only by the factor $t^{(\al-1)a_+}$, and
$|a_+|$ may be small; in Case A2 the relative error of $L_t$ is bounded by
a multiple of $t^{-(a_s+a_+)}+t^{-1-a_s}$, to which the contribution of
$\eps_t$ must be added when $\al<2$.

\emph{Numerical method.} We took $s=1$ and $H(t)\equiv H_+$ for the second
kernel. The integral over $u\in(0,s)$ was computed directly, and the
integral over $u=-v<0$ in the variable $x=\log v$ on $[-40,45]$, split into
unit intervals; on each interval we used the adaptive Gauss-Kronrod routine
of QUADPACK (through \texttt{scipy.integrate.quad}, double precision,
relative tolerance $10^{-10}$). By Lemma~\ref{lem:phi} and Lemma~\ref{lem:ineq}(iv),
the integrand is bounded by a power of $v$ near $0$ and near $\infty$, so the
neglected parts $v<e^{-40}$ and $v>e^{45}$ are bounded by multiples of
$e^{-40c}$ and $e^{-45c'}$ for some $c,c'>0$ depending on the parameters. For $|u|\ge2t$ the kernel was evaluated
by the series of Lemma~\ref{lem:kernel-asymp}(ii), which avoids the
cancellation in \eqref{eq:kernel}. The constants $C(s)$ were computed from
\eqref{eq:C1}--\eqref{eq:C4} and \eqref{eq:K-closed} in the same way. The
Python scripts are available from the author upon request.

In the following table, the last three columns give the ratio
$(\tau(s,t)-c_0(s))/(C(s)\,t^{\rho(s)})$, which tends to $1$ by
Theorem~\ref{thm:codiff-asymp}; an empty entry means that the value was not
computed. In the row marked $*$ ($n=1$, $\al<2$, regime E4),
$C(s)=C_4(s)+D(s)$ as in part (b2).
\begin{center}
\begin{tabular}{@{}lccccc@{}}
\toprule
& & & \multicolumn{3}{c}{ratio $(\tau(s,t)-c_0(s))/(C(s)\,t^{\rho(s)})$}\\
\cmidrule(l){4-6}
$(\al,n,H(s),H_+)$ & regime, case & $\rho(s)$ & $t=10^4$ & $t=10^8$ & $t=10^{12}$\\
\midrule
$(1.5,1,0.90,0.85)$ & E1, A1 & $0.325$ & $1.0002$ & & \\
$(1.5,3,2.90,2.85)$ & E1, A1 & $1.325$ & $1.005$ & & \\
$(1.5,2,1.55,1.95)$ & E1, A2 & $0.525$ & $0.939$ & & \\
$(1.5,1,0.4667,0.9667)$ & E1, A2 & $-0.050$ & $0.716$ & $0.888$ & $0.955$\\
$(1.5,3,2.45,2.80)$ & E4, A3 & $0.933$ & $0.945$ & $0.997$ & \\
$(1.5,1,0.45,0.80)^*$ & E4, A3 & $-0.067$ & $0.912$ & $0.996$ & $1.007$\\
$(1.5,2,1.50,1.55)$ & E2, B2 & $0.50$ & $1.400$ & $1.318$ & $1.215$\\
$(2,1,0.3,0.3)$ & E2, B2 & $0$ & $1.015$ & $1.0004$ & \\
$(2,2,1.8,1.6)$ & E1, A1 & $1.4$ & $1.055$ & $1.001$ & \\
\bottomrule
\end{tabular}
\end{center}
In the regime E2 row with $\al=1.5$, the ratio including the second term
$\al\Ga_+^{1-\al}m_s t^{(\al-1)\ga_+}$ is $0.804$, $0.920$ and $0.969$ at the
same values of $t$. For $\al=2$ the quadrature agrees with the exact
covariance of fractional Brownian motion to eight digits.
\end{remark}

\section{Long-range dependence of the $n$-th order increments}
\label{sec:increments}

Corollary~\ref{cor:no-decay} shows that the codifference of the process
itself is not the right object to measure memory. In this section we study the
$n$-th order increments with unit step
\[
Z_k:=\De^n_1X^{(n)}_H(k)=\sum_{j=0}^nc_jX^{(n)}_H(k+j),
\qquad k\in\mathbb Z ,
\]
with $c_j$ as in \eqref{eq:Delta-n}; we write $\De^n:=\De^n_1$.

\emph{Scope.} Except in Lemma~\ref{lem:G} and in
Subsection~\ref{subsec:functional-H}, the Hurst function is constant in this
section, $H\equiv H_0$, so that $(Z_k)$ is the increment sequence of the
$n$-FSM $Y^{(n)}_{H_0}$. For a non-constant Hurst function our results only
cover the case where $H$ is constant after some time
(Proposition~\ref{prop:eventually-constant}); the reason is explained in
Remark~\ref{rem:open-incr}.

\subsection{The increment kernel}

For $\ga>-1/\al$ we define
\begin{equation}
\label{eq:G-def}
G_\ga(x):=\frac{1}{\Ga(\ga+1)}\sum_{j=0}^n c_j\,(x+j)_+^{\ga},\qquad x\in\R .
\end{equation}

\begin{lemma}[Increment kernel]
\label{lem:G}
Let $k\in\R$ and assume that $\ga(k)=\ga(k+1)=\cdots=\ga(k+n)=:\ga$. Then
\begin{equation}
\label{eq:Y-kernel}
\Delta^n X^{(n)}_H(k)=\int_\R G_\ga(k-u)\,dM_\al(u)\quad\text{a.s.}
\end{equation}
Assume moreover $a:=\ga-n+1\in(-1/\al,1-1/\al)$ (which holds when
$\ga=\ga(k)$ under Assumption~\ref{ass:params}) and $a\neq0$. Then:
\begin{enumerate}[label=(\roman*)]
\item $G_\ga(x)=0$ for $x\le-n$;
\item for $x>0$,
\begin{equation}
\label{eq:G-repr}
G_\ga(x)=\frac{1}{\Ga(a)}\int_{[0,1]^n}\bigl(x+\theta_1+\cdots+\theta_n\bigr)^{a-1}d\theta ;
\end{equation}
\item for $x>0$, $G_\ga(x)$ has the sign of $a$, and
\begin{equation}
\label{eq:G-bounds}
(x+n)^{a-1}\le|\Ga(a)\,G_\ga(x)|\le x^{a-1},
\qquad
\Bigl|G_\ga(x)-\frac{x^{a-1}}{\Ga(a)}\Bigr|\le\frac{n(1-a)}{2|\Ga(a)|}\,x^{a-2};
\end{equation}
\item for $x\in(-n,1]$, $|G_\ga(x)|\le\Ga(\ga+1)^{-1}\sum_{j=0}^n\binom nj(x+j)_+^{\ga}$; in
particular $|G_\ga|^p$ is integrable on $(-n,1]$ for every
$p\in(0,\al]$;
\item $G_\ga\in L^\al(\R)$, and $\psi(G_\ga)\in L^1(\R)$ if and only if
$(1-a)(\al-1)>1$.
\end{enumerate}
\end{lemma}

\begin{proof}
\emph{Kernel.} The stochastic integral is linear, so the kernel
of $\Delta^n X^{(n)}_H(k)$ is the function $u\mapsto\sum_{j}c_jf_n(k+j,u)$. Since all
the exponents are equal to $\ga$,
\[
\Ga(\ga+1)\sum_{j=0}^n c_jf_n(k+j,u)
=\sum_{j=0}^n c_j(k+j-u)_+^{\ga}
-\sum_{i=0}^{n-1}\binom{\ga}{i}(-u)_+^{\ga-i}\Bigl[\sum_{j=0}^n c_j(k+j)^i\Bigr].
\]
For $i\le n-1$, $j\mapsto(k+j)^i$ is a polynomial of degree $i<n$, and its
$n$-th forward difference $\sum_jc_j(k+j)^i$ vanishes. Hence the kernel is
$\Ga(\ga+1)^{-1}\sum_jc_j(k-u+j)_+^{\ga}=G_\ga(k-u)$.

\emph{(i)} If $x\le-n$, then $x+j\le0$ for all $j\le n$.

\emph{(ii)} Put $q(y):=y^{\ga}/\Ga(\ga+1)$ for $y>0$. For $x>0$, $q$ is smooth
on $(x,\infty)$, and the identity
$\Delta f(x)=f(x+1)-f(x)=\int_0^1f'(x+\theta)\,d\theta$, iterated $n$ times,
gives
\[
G_\ga(x)=\Delta^n q(x)=\int_{[0,1]^n}q^{(n)}(x+\theta_1+\cdots+\theta_n)\,d\theta .
\]
Since $q^{(n)}(y)=\frac{\ga(\ga-1)\cdots(\ga-n+1)}{\Ga(\ga+1)}y^{\ga-n}
=\frac{y^{\ga-n}}{\Ga(\ga-n+1)}=\frac{y^{a-1}}{\Ga(a)}$, we get \eqref{eq:G-repr}.

\emph{(iii)} $\Ga(a)$ has the sign of $a$ for $a\in(-1,1)\setminus\{0\}$. Since
$a-1<0$ and $0\le\sum\theta_i\le n$, the integrand in \eqref{eq:G-repr} lies
between $(x+n)^{a-1}$ and $x^{a-1}$, which gives the first bound. For the
second, by the mean value theorem
$|(x+\sigma)^{a-1}-x^{a-1}|\le(1-a)\sigma x^{a-2}$ for $\sigma\ge0$, and
$\int_{[0,1]^n}(\theta_1+\cdots+\theta_n)\,d\theta=n/2$.

\emph{(iv)} The bound is the triangle inequality. Each $(x+j)_+^{\ga p}$ is
integrable on $(-n,1]$ when $\ga p>-1$, which holds for $p\le\al$ since
$\ga>-1/\al$ (if $\ga<0$), and trivially if $\ga\ge0$.

\emph{(v)} By (iii), $|G_\ga(x)|^p\asymp x^{p(a-1)}$ as $x\to\infty$, which is
integrable at infinity if and only if $p(1-a)>1$. For $p=\al$ this holds
because $a<1-1/\al$. For $p=\al-1$ it is the stated condition. Local
integrability follows from (iv).
\end{proof}

\begin{remark}[Stationarity]
\label{rem:stationarity}
If $H\equiv H_0$ is constant, Lemma~\ref{lem:G} shows that
$Z_k=\int_\R G_{\ga_0}(k-u)\,dM_\al(u)$. Since the Lebesgue control measure is
translation invariant, $(Z_k)_{k\in\mathbb Z}$ is a stationary S$\al$S moving
average (see also Theorem~\ref{thm:existence}(iii) for general steps). If
$a=0$, i.e.\ $H_0=n-1+1/\al$, then $\ga_0=n-1$ and $G_{\ga_0}$ vanishes on
$(0,\infty)$ (the $n$-th difference of a polynomial of degree $n-1$), so
$G_{\ga_0}$ is supported in $[-n,0]$, $Z_k$ depends only on the restriction of
$M_\al$ to $[k,k+n]$, and $(Z_k)$ is $(n-1)$-dependent: $r(m)=0$ for $m\ge n$.
We exclude this case below.
\end{remark}

\subsection{Asymptotics of the increment codifference}

Throughout this subsection $H\equiv H_0\in(n-1,n)$ is constant,
$\ga:=H_0-1/\al$, $a:=\ga-n+1\neq0$, $G:=G_\ga$, and
\[
r(m):=\tau(Z_0,Z_m)=\int_\R\Phi\bigl(G(-u),G(m-u)\bigr)du
=\int_{-n}^{\infty}\Phi\bigl(G(x),G(x+m)\bigr)dx,
\]
where we substituted $x=-u$ and used Lemma~\ref{lem:G}(i). We set
\[
a_*:=\frac{\al-2}{\al-1}\le0,\qquad
H_*:=n-\frac{1}{\al(\al-1)}=a_*+n-1+\frac1\al,
\]
so that $a>a_*$ if and only if $H_0>H_*$. Note that $H_*>n-1$ (equivalently
$a_*>-1/\al$) if and only if $\al^2-\al-1>0$, that is,
$\al>\frac{1+\sqrt5}{2}$.

\begin{lemma}[Sign of $\Psi$ for $n=1$]
\label{lem:Psi}
Let $n=1$, $\al\in(1,2)$, $a\in(-1/\al,1-1/\al)$ with $(1-a)(\al-1)>1$ (so
that $a<a_*\le0$). Then
\[
\Psi:=\int_\R\psi\bigl(G(x)\bigr)\,dx<0 .
\]
\end{lemma}

\begin{proof}
For $n=1$ and $\ga=a<0$, $\Ga(a+1)G(x)=(x+1)^a$ on $(-1,0)$ and
$\Ga(a+1)G(x)=-N(x)$ on $(0,\infty)$, with $N(x):=x^a-(x+1)^a>0$. Put
$P(y):=y^a$ on $(0,1)$ and $\beta:=\al-1\in(0,1)$. After the shift $y=x+1$,
\[
\Ga(a+1)^{\beta}\,\Psi=\int_0^1P(y)^{\beta}dy-\int_0^\infty N(x)^{\beta}dx .
\]
Both integrals are finite: the first because $a\beta>-1$, the second because
$N(x)\sim|a|x^{a-1}$ and $\beta(1-a)>1$. We show that the second is larger.

\emph{Distribution functions.} For a function $f\ge0$ we write
$\mu_f(\lambda):=|\{f>\lambda\}|$.
Then $\int f^{\beta}=\int_0^\infty\beta\lambda^{\beta-1}\mu_f(\lambda)\,d\lambda$
and $\int f=\int_0^\infty\mu_f$. Since $a<0$,
$\mu_P(\lambda)=\min(1,\lambda^{1/a})$. Since $0<N(x)<x^a$,
$\{N>\lambda\}\subset\{x<\lambda^{1/a}\}$, so $\mu_N(\lambda)\le\lambda^{1/a}$.
Moreover, $N$ is continuous and decreasing, so $\{N>\lambda\}=(0,x_\lambda)$
with $N(x_\lambda)=\lambda$, and $x_\lambda<\lambda^{1/a}$ because
$N(\lambda^{1/a})<\lambda$. Hence
\begin{equation}
\label{eq:D-neg}
\mathcal D(\lambda):=\mu_N(\lambda)-\mu_P(\lambda)<0\qquad\text{for }\lambda>1 .
\end{equation}

\emph{Equal masses.} $\int_0^1P=\frac{1}{1+a}$, and for $X>0$,
$\int_0^XN=\frac{X^{a+1}-(X+1)^{a+1}+1}{a+1}\to\frac{1}{1+a}$ as $X\to\infty$,
since $0<a+1<1$. Therefore $\int_0^\infty \mathcal D(\lambda)\,d\lambda=0$, and by
\eqref{eq:D-neg}, $\int_0^1\mathcal D=-\int_1^\infty\mathcal D>0$.

\emph{Chebyshev's inequality.} On $(0,1)$, $\mu_P\equiv1$ and $\mu_N$ is
non-increasing, so $\mathcal D$ is non-increasing; $\lambda^{\beta-1}$ is
non-increasing too. Both are integrable on $(0,1)$, and so is their product,
since $\int_0^1\mu_N\le\int N<\infty$ and
$\int_0^1\lambda^{\beta-1}\mu_N(\lambda)d\lambda\le\beta^{-1}\int N^\beta<\infty$.
Chebyshev's integral inequality for similarly ordered functions on $[0,1]$
gives
\[
\int_0^1\lambda^{\beta-1}\mathcal D(\lambda)\,d\lambda
\ge\int_0^1\lambda^{\beta-1}d\lambda\int_0^1\mathcal D(\lambda)\,d\lambda
=\frac1\beta\int_0^1\mathcal D .
\]
On $(1,\infty)$, $\lambda^{\beta-1}\le1$ and $\mathcal D<0$, so
$\lambda^{\beta-1}\mathcal D(\lambda)\ge \mathcal D(\lambda)$. Adding,
\[
\int_0^\infty\lambda^{\beta-1}\mathcal D
\ge\frac1\beta\int_0^1\mathcal D+\int_1^\infty\mathcal D
=\Bigl(\frac1\beta-1\Bigr)\int_0^1\mathcal D>0 .
\]
Thus $\int N^\beta-\int P^\beta=\beta\int_0^\infty\lambda^{\beta-1}\mathcal D>0$, i.e.\ $\Psi<0$.
\end{proof}

\begin{theorem}[Codifference of the $n$-th order increments; Kawai {\cite[Theorem~5.2]{Kawai2016}}]
\label{thm:incr}
Let $\al\in(1,2)$, $n\in\N^*$, $H\equiv H_0\in(n-1,n)$ constant, with
$a\neq0$ and $a\neq a_*$.
\begin{enumerate}[label=(\roman*)]
\item If $H_0>H_*$, then, as $m\to\infty$,
\begin{equation}
\label{eq:incr-L}
r(m)\sim D_L\,m^{\al(H_0-n)},\qquad
D_L:=|\Ga(a)|^{-\al}\int_0^\infty\Phi\bigl(w^{a-1},(1+w)^{a-1}\bigr)dw>0 .
\end{equation}
\item If $H_0<H_*$ (possible only if $\al>\frac{1+\sqrt5}2$), then
\begin{equation}
\label{eq:incr-S}
r(m)=\frac{\al\,\Psi}{\Ga(a)}\,m^{H_0-n-1/\al}+o\bigl(m^{H_0-n-1/\al}\bigr),
\qquad \Psi=\int_\R\psi\bigl(G(x)\bigr)dx .
\end{equation}
For $n=1$, $\Psi<0$ by Lemma~\ref{lem:Psi}, so
$D_S:=\al\Psi/\Ga(a)>0$ and $r(m)\sim D_S\,m^{H_0-1-1/\al}$.
\end{enumerate}
Part~(i) also holds for $\al=2$ when $H_0>n-\frac12$ (that is, $a>0=a_*$).
\end{theorem}

\begin{remark}
The threshold $H_*$ and the two exponents are those of
\cite[Theorem~5.2 and Remark~5.3]{Kawai2016}, who also covers
$\al\in(0,1]$; note that the codifference $\tau_n$ in \cite{Kawai2016} is
defined with the opposite sign. For $n=1$ they are the results obtained for
linear fractional stable noise by Astrauskas, Levy and Taqqu
\cite{AstrauskasLevyTaqqu1991}. What we add is a self-contained proof in our
framework, the positivity of $D_L$, the sign of $\Psi$ for $n=1$, and the
boundary cases of Corollary~\ref{cor:LRD-incr}; these estimates are also
those needed for Proposition~\ref{prop:eventually-constant}. The $n$-th order
increments behave, for large lags, like linear fractional stable noise with
Hurst index $H_0-(n-1)$. In case (i) the
constant $D_L$ depends only on $(\al,a)$, so the comparison also holds at
the level of constants; in case (ii) the constants differ through
$\Psi=\Psi_n$.
\end{remark}

\begin{proof}[Proof of Theorem~\ref{thm:incr}]
We use the symmetry $\Phi(A,B)=\Phi(B,A)$ and the homogeneity
property $\Phi(\lambda A,\lambda B)=|\lambda|^\al\Phi(A,B)$, together with the bound of
Lemma~\ref{lem:ineq}(iv), written with the roles of the arguments exchanged:
\begin{equation}
\label{eq:Phi-swap}
|\Phi(A,B)|\le(1+\al)|B|^\al+\al|B|\,|A|^{\al-1}.
\end{equation}
We also write $G(y)=\Ga(a)^{-1}y^{a-1}\bigl(1+\eta(y)\bigr)$ for $y>0$; by
\eqref{eq:G-bounds}, $|\eta(y)|\le\frac{n(1-a)}{2y}\le\frac n{y}$.

Fix $\delta\in(0,1]$ and let $m\ge m_0(\delta):=2(n+1)/\delta$. We split
\[
r(m)=I_1+I_2,\quad
I_1:=\int_{-n}^{\delta m}\Phi\bigl(G(x),G(x+m)\bigr)dx,\quad
I_2:=\int_{\delta m}^{\infty}\Phi\bigl(G(x),G(x+m)\bigr)dx .
\]
For $x\ge-n$ we have $x+m\ge m-n\ge m/2\ge1$, so by \eqref{eq:G-bounds}
\begin{equation}
\label{eq:B-m}
|G(x+m)|\le|\Ga(a)|^{-1}(m/2)^{a-1}=:B_m,\qquad x\ge-n .
\end{equation}

\medskip
\noindent\emph{The far region $I_2$.} Substituting $x=mw$,
\begin{align*}
I_2&=m\int_\delta^\infty\Phi\bigl(G(mw),G(m(1+w))\bigr)dw\\
&=|\Ga(a)|^{-\al}m^{\al(a-1)+1}\int_\delta^\infty
\Phi\bigl(w^{a-1}(1+\eta_1),(1+w)^{a-1}(1+\eta_2)\bigr)dw,
\end{align*}
with $\eta_1=\eta(mw)$, $\eta_2=\eta(m(1+w))$, and
$|\eta_i|\le\frac{n}{\delta m}\le\frac12$. As $m\to\infty$, $\eta_i\to0$ for each
$w$, and $\Phi$ is continuous. Since
$|\Phi(A,B)|\le|A|^\al+|B|^\al+|A-B|^\al\le3(|A|+|B|)^\al$, the integrand is
bounded on $[\delta,\infty)$ by
$3\bigl(\tfrac32w^{a-1}+\tfrac32(1+w)^{a-1}\bigr)^\al\le3^{\al+1}w^{\al(a-1)}$,
which is integrable at infinity because $\al(1-a)>1$. By dominated convergence,
\begin{equation}
\label{eq:I2}
\frac{I_2}{m^{\al(a-1)+1}}\longrightarrow
|\Ga(a)|^{-\al}\int_\delta^\infty\Phi\bigl(w^{a-1},(1+w)^{a-1}\bigr)dw .
\end{equation}
Note that $\al(a-1)+1=\al(H_0-n)$.

\medskip
\noindent\emph{Proof of (i).} Here $a>a_*$, i.e.
$\kappa:=1-(1-a)(\al-1)>0$. By \eqref{eq:Phi-swap} and \eqref{eq:B-m},
\[
|I_1|\le(1+\al)B_m^\al(n+\delta m)+\al B_m\int_{-n}^{\delta m}|G(x)|^{\al-1}dx .
\]
By Lemma~\ref{lem:G}(iv) and \eqref{eq:G-bounds}, with
$C_0:=\int_{-n}^1|G|^{\al-1}<\infty$,
\[
\int_{-n}^{\delta m}|G|^{\al-1}
\le C_0+|\Ga(a)|^{1-\al}\int_1^{\delta m}x^{(a-1)(\al-1)}dx
\le C_0+\frac{|\Ga(a)|^{1-\al}}{\kappa}(\delta m)^{\kappa}.
\]
Since $B_m\le Cm^{a-1}$ and $(a-1)+\kappa=\al(a-1)+1$, we obtain
\[
|I_1|\le C\Bigl[\delta\,m^{\al(a-1)+1}+m^{\al(a-1)}+m^{a-1}
+\delta^{\kappa}m^{\al(a-1)+1}\Bigr].
\]
As $m^{a-1}=o(m^{\al(a-1)+1})$ (because $\kappa>0$),
$\limsup_m|I_1|/m^{\al(a-1)+1}\le C(\delta+\delta^\kappa)$, where $C$ depends
only on $(n,\al,a)$ and not on $\delta$. Combined with
\eqref{eq:I2},
\[
\limsup_{m\to\infty}\Bigl|\frac{r(m)}{m^{\al(a-1)+1}}-D_L\Bigr|
\le C(\delta+\delta^\kappa)
+|\Ga(a)|^{-\al}\int_0^\delta\bigl|\Phi\bigl(w^{a-1},(1+w)^{a-1}\bigr)\bigr|dw .
\]
The last integral is finite and tends to $0$ with $\delta$: by
\eqref{eq:Phi-swap} with $A=w^{a-1}$ and $B=(1+w)^{a-1}\le1$, the integrand is
at most $(1+\al)+\al w^{(a-1)(\al-1)}$, and $(a-1)(\al-1)=\kappa-1>-1$. Letting
$\delta\to0$ proves \eqref{eq:incr-L}. Finally, for $A,B>0$,
$|A-B|<\max(A,B)$, hence $\Phi(A,B)>A^\al+B^\al-\max(A,B)^\al=\min(A,B)^\al>0$,
so $D_L>0$. Nothing in this argument uses $\al<2$; it only uses $\kappa>0$,
which for $\al=2$ means $a>0$.

\medskip
\noindent\emph{Proof of (ii).} Here $(1-a)(\al-1)>1$, so $\psi(G)\in L^1$ by
Lemma~\ref{lem:G}(v). Take $\delta=1$. By \eqref{eq:I2},
$I_2=O(m^{\al(a-1)+1})$. For $I_1$ we write, using Lemma~\ref{lem:ineq}(ii)
with the roles of the arguments exchanged,
\[
\Phi(A,B)=\al\,\psi(A)\,B+|B|^\al+R(B,A),\qquad|R(B,A)|\le2^{2-\al}|B|^\al .
\]
Hence, by \eqref{eq:B-m},
\begin{align*}
I_1&=\al\int_{-n}^{m}\psi\bigl(G(x)\bigr)G(x+m)\,dx+O\bigl((n+m)B_m^\al\bigr)\\
&=\al\int_{-n}^{m}\psi\bigl(G(x)\bigr)G(x+m)\,dx+O\bigl(m^{\al(a-1)+1}\bigr).
\end{align*}
For fixed $x\ge-n$,
\[
m^{1-a}G(x+m)=\frac{1}{\Ga(a)}\Bigl(1+\frac xm\Bigr)^{a-1}\bigl(1+\eta(x+m)\bigr)
\longrightarrow\frac{1}{\Ga(a)},
\]
and $|m^{1-a}G(x+m)|\le|\Ga(a)|^{-1}2^{1-a}$ by \eqref{eq:B-m}. Since
$\psi(G)\in L^1$, dominated convergence gives
\[
m^{1-a}\int_{-n}^{m}\psi\bigl(G(x)\bigr)G(x+m)\,dx\longrightarrow\frac{\Psi}{\Ga(a)} .
\]
Finally $\al(a-1)+1<a-1$, because this is equivalent to $(\al-1)(1-a)>1$. So
all remainders are $o(m^{a-1})$, and \eqref{eq:incr-S} follows, with
$a-1=H_0-n-1/\al$. The statement for $n=1$ follows from Lemma~\ref{lem:Psi}
and $\Ga(a)<0$.
\end{proof}

\begin{definition}[Long-range dependence]
\label{def:LRD}
A stationary S$\al$S sequence $(Z_k)_{k\in\mathbb Z}$ with codifference
$r(m)=\tau(Z_0,Z_m)$ is said to have \emph{long-range dependence} if
$\sum_{m\ge1}|r(m)|=\infty$, and \emph{short-range dependence} otherwise.
\end{definition}

Several non-equivalent definitions of long-range dependence are used for
stable processes; Definition~\ref{def:LRD} is the codifference criterion,
which reduces to the summability of the covariances when $\al=2$ (since then
$r(m)=\mathrm{Cov}(Z_0,Z_m)$ by Lemma~\ref{lem:codiff-finite}(iv)). For the
increments of $n$-FSM it does not matter whether one requires
$\sum_m|r(m)|=\infty$ or $\sum_mr(m)=\infty$: whenever the series of
$|r(m)|$ diverges (that is, $H_0\ge n-\frac1\al$, see
Corollary~\ref{cor:LRD-incr}), we are in case (i) of Theorem~\ref{thm:incr},
where $r(m)\sim D_Lm^{\al(H_0-n)}$ with $D_L>0$, so that $r(m)>0$ for all large
$m$. In the remaining cases the series of $|r(m)|$ converges, whatever the
sign of $r(m)$. In particular the sign of $\Psi$ is not needed for the
criterion; it only determines the eventual sign of $r(m)$ in case (ii). The
relation with the moving-average notion of memory, based on the decay of
the kernel, is discussed in Remark~\ref{rem:two-notions}.

\begin{corollary}[Long-range dependence of the increments; compare {\cite[Corollary~5.4]{Kawai2016}}]
\label{cor:LRD-incr}
Let $\al\in(1,2)$ and $H\equiv H_0\in(n-1,n)$ with $a\neq0$, that is,
$H_0\neq n-1+1/\al$. Then
\[
\sum_{m\ge1}|r(m)|=\infty
\quad\Longleftrightarrow\quad
H_0\ge n-\frac1\al ,
\]
and $r(m)>0$ for all large $m$ when $H_0\ge n-\frac1\al$. More precisely:
\begin{enumerate}[label=(\roman*)]
\item if $H_0>n-\frac1\al$, then $r(m)\sim D_Lm^{\al(H_0-n)}$ with
$-1<\al(H_0-n)<0$;
\item if $H_0=n-\frac1\al$, then $r(m)\sim D_L\,m^{-1}$, so the series diverges
logarithmically;
\item if $H_*<H_0<n-\frac1\al$, then $r(m)\sim D_Lm^{\al(H_0-n)}$ with
$\al(H_0-n)<-1$;
\item if $H_0=H_*$ (possible only if $\al>\frac{1+\sqrt5}2$), then
$|r(m)|\le C\,m^{-1/(\al-1)}\log m$, with $\frac1{\al-1}>1$;
\item if $H_0<H_*$, then $r(m)=O(m^{H_0-n-1/\al})$, with $H_0-n-\frac1\al<-1$.
\end{enumerate}
\end{corollary}

\begin{proof}
Since $\frac1\al<\frac{1}{\al(\al-1)}$, we have $n-\frac1\al>H_*$, so the
region $H_0\ge n-\frac1\al$ lies in case (i), where
$|r(m)|\sim D_Lm^{\al(H_0-n)}$ with $D_L>0$; the series diverges if and only
if $\al(H_0-n)\ge-1$. If $H_*<H_0<n-\frac1\al$, then $\al(H_0-n)<-1$ and the
series converges. If $H_0<H_*$, then $a<a_*<0$ and, by
Theorem~\ref{thm:incr}(ii), $r(m)=O(m^{a-1})$ with $a-1<-1$, whatever the
value of $\Psi$; the series converges.

On the boundary $H_0=H_*$, i.e.\ $a=a_*$, we have $\kappa=0$ and
$a-1=\al(a-1)+1=-\frac{1}{\al-1}<-1$. We repeat the proof of
Theorem~\ref{thm:incr}(i) with $\delta=1$ and $m\ge2(n+1)$. The far-region
argument only uses $\al(1-a)>1$, which holds since
$\al(1-a_*)=\frac{\al}{\al-1}>1$; hence $I_2=O(m^{\al(a-1)+1})$. In the bound
for $I_1$, $\int_1^m x^{(a-1)(\al-1)}dx=\log m$ because $(a-1)(\al-1)=-1$, so
\[
|I_1|\le C\bigl[m^{\al(a-1)+1}+m^{\al(a-1)}+m^{a-1}(1+\log m)\bigr].
\]
Therefore $|r(m)|\le C\,m^{-1/(\al-1)}\log m$, which is summable because
$\frac{1}{\al-1}>1$.
\end{proof}

\begin{remark}[Two notions of memory]
\label{rem:two-notions}
The kernel $G$ decays as $x^{d-1}$ with $d:=a=H_0-(n-1)-\frac1\al$, so in the
moving-average sense the increments have long memory if and only if $d>0$,
that is, $H_0>n-1+\frac1\al$. For $\al<2$ we have
$n-\frac1\al<n-1+\frac1\al$, so Corollary~\ref{cor:LRD-incr} shows that the
codifference is non-summable on the larger range $H_0\ge n-\frac1\al$
(excluding the degenerate value $H_0=n-1+\frac1\al$). In particular, for
$n-\frac1\al\le H_0<n-1+\frac1\al$ the kernel is integrable while the
codifference is not summable. As for linear fractional stable noise ($n=1$),
the two notions do not coincide when $\al<2$. In the classical Gaussian
case $\al=2$, where the covariance of the increments is of order
$m^{2H_0-2n}$, both give $H_0>n-\frac12$.
\end{remark}

\subsection{Functional Hurst parameter}
\label{subsec:functional-H}

\begin{proposition}
\label{prop:eventually-constant}
Let Assumption~\ref{ass:params} hold and assume that $H(t)=H_+$ for all
$t\ge T$. Then for every integer $k\ge T$,
\[
\Delta^nX^{(n)}_H(k)=\int_\R G_{\ga_+}(k-u)\,dM_\al(u),
\]
and the sequence $(\Delta^nX^{(n)}_H(k))_{k\ge T}$ is stationary, with the same
law as the $n$-th order increments of $n$-FSM with Hurst parameter $H_+$.
Consequently, for $k\ge T$ and $m\ge0$,
$\tau\bigl(\Delta^nX^{(n)}_H(k),\Delta^nX^{(n)}_H(k+m)\bigr)=r(m)$ exactly, where
$r$ is the increment codifference of Section~\ref{sec:increments} with
$H_0=H_+$; Theorem~\ref{thm:incr} and Corollary~\ref{cor:LRD-incr} apply.
\end{proposition}

\begin{proof}
For $k\ge T$ we have $\ga(k)=\cdots=\ga(k+n)=\ga_+$, and Lemma~\ref{lem:G}
gives the kernel. Stationarity follows from the translation invariance of the
control measure of $M_\al$.
\end{proof}

\begin{remark}[Hurst functions that are not eventually constant]
\label{rem:open-incr}
Assume only that $H(t)\to H_+$. For $k\ge0$ the kernel of
$Z_k=\De^nX^{(n)}_H(k)$ is $\sum_jc_jf_n(k+j,\cdot\,;\ga(k+j))$, and we can write
it as $G_{\ga_+}(k-\cdot)+\mathcal E_k$ with
\[
\mathcal E_k(u):=\sum_{j=0}^nc_j\bigl[f_n(k+j,u;\ga(k+j))-f_n(k+j,u;\ga_+)\bigr].
\]
The identity of Lemma~\ref{lem:G} relies on the exact cancellation of the
polynomial terms $\binom{\ga}{i}(k+j)^i(-u)_+^{\ga-i}$, $i\le n-1$, in the
$n$-th difference, which requires the same exponent $\ga$ for the $n+1$
times $k,\ldots,k+n$. When $\ga(k),\ldots,\ga(k+n)$ differ, these terms no
longer cancel, and $\mathcal E_k$ must be estimated directly. For a frozen
exponent the kernel is homogeneous, $f_n(t,u;\ga)=t^{\ga}f_n(1,u/t;\ga)$ for
$t>0$; differentiating in $\ga$ and using Lemma~\ref{lem:deriv-bounds} with
$T=\{1\}$, we get
$\|\partial_\ga f_n(t,\cdot\,;\ga)\|_{L^\al}\le C\,t^{\ga+1/\al}(1+\log t)$ for
$t\ge1$ and $\ga$ in a compact interval. Hence, by the mean value theorem in
$\ga$, with $\eps_k:=\max_{0\le j\le n}|H(k+j)-H_+|$,
\[
\|\mathcal E_k\|_{L^\al}\le C\,\eps_k\,k^{H_++\eps_k}(1+\log k),
\]
while $\|G_{\ga_+}\|_{L^\al}$ does not depend on $k$. This bound makes
$\mathcal E_k$ negligible only if $H(k)-H_+$ decays faster than
$k^{-H_+}$, a polynomial rate that depends on $n$ since $H_+>n-1$.
An extension of Theorem~\ref{thm:incr} would moreover require comparing the
contributions of $\mathcal E_k$ and $\mathcal E_{k+m}$ to the codifference
with $r(m)$. We
have not carried out this analysis, and we do not know whether
Corollary~\ref{cor:LRD-incr} holds with $H_0=H_+$ in this generality.
\end{remark}

\begin{remark}[Numerical evaluation]
All the computations below are numerical evidence, not proofs. We
computed $r(m)$ for $\al\in\{1.5,1.7,1.8,1.9\}$, $n\in\{1,2,3\}$ and several
values of $a$ in both regimes, for $m$ up to $10^7$, with the adaptive
Gauss-Kronrod routine of QUADPACK (\texttt{scipy.integrate.quad}, double
precision): the integral over $x\in(-n,4n)$ is split at the integers, where
$G$ is singular, and the integral over $(4n,10^6m)$ is split into $80$
logarithmically spaced intervals; for $x>4n$, $G$ is evaluated by its
binomial series, which avoids the cancellation in \eqref{eq:G-def}, and the
part beyond $10^6m$ is replaced by the integral of the leading term of the
integrand.
The relative error decays like $m^{-|\kappa|}$, $\kappa=1-(1-a)(\al-1)$, which
can be slow when $|\kappa|$ is small (for instance the ratio is $0.74$ at
$m=10^7$ for $(\al,n,a)=(1.8,2,-0.1)$, where $\kappa=0.12$); a Richardson
extrapolation with this rate gives the limit $1$ to within $10^{-4}$ in all
cases computed. For $\Psi$, we used the same method on
$(-n,10^9)$, plus the exact integral over $(10^9,\infty)$ of the leading term
$\sign(a)|\Ga(a)|^{1-\al}x^{(a-1)(\al-1)}$, on the grid of the $13$ equally
spaced values of $\al$ in $[1.63,1.99]$ and the $13$ interior points of the
uniform subdivision of $[-\frac1\al,a_*]$ into $14$ intervals, for
$n=1,2,3,4$ ($169$ points for each $n$); the estimated quadrature error was
below $2\cdot10^{-10}$ at every point. All the computed values are negative.
For $n=1,2,3,4$ they lie in $[-373.9,-0.029]$, $[-373.9,-0.027]$,
$[-373.8,-0.024]$ and $[-373.7,-0.023]$ respectively; the values of largest
modulus occur next to $a_*$, where $\Psi\to-\infty$ because $\psi(G)$ is not
integrable at $a=a_*$, and those of smallest modulus for $\al=1.99$ near
$a=-\frac1\al$. We also checked the points $a=-\frac1\al+\eps$,
$\eps\in\{10^{-2},10^{-3},10^{-4}\}$, for $\al\in\{1.7,1.9,1.99\}$: the
values stabilize as $\eps\to0$ (for instance $\Psi\approx-0.021$ for $n=4$,
$\al=1.99$) and remain negative. This supports the conjecture that $\Psi<0$
for every $n$, but we have no proof for $n\ge2$. At
the boundary $H_0=H_*$, the numerics are consistent with
$r(m)\sim\al|\Ga(a_*)|^{-\al}m^{-1/(\al-1)}\log m$; for $n=1$ an asymptotic of
this form is proved in \cite[Theorem~3.1]{LevyTaqqu2012}.
\end{remark}

\section{Conclusion and open problems}
\label{sec:conclusion}

We introduced the $n$-th order multifractional stable motion, which combines
$\al$-stable heavy tails, a functional Hurst parameter $H(t)\in(n-1,n)$ and a
Taylor regularization of order $n$. When $H$ is of class $\mathcal C^n$, its
$n$-th order increments are locally asymptotically self-similar, the tangent
process being the $n$-th order increments of the $n$-FSM of index $H(t_0)$.
The codifference of the process itself grows or converges to a non-zero
limit, at a rate that we determined in three regimes (except, for $n=1$ and
$\al<2$, in part of regime E4), so that memory has to be measured on the
increments. For constant $H\neq n-1+1/\al$ and
$\al\in(1,2)$, the codifference of the $n$-th order increments behaves like
that of linear fractional stable noise with Hurst index $H-(n-1)$, with the
threshold $n-1/(\al(\al-1))$ \cite{Kawai2016}, and it is non-summable if and
only if $H\ge n-1/\al$; this criterion extends to $n$-MFSM with eventually
constant Hurst function.

The following questions are not answered in this paper. We first list
those that are needed to complete the theory developed here.
\begin{enumerate}[label=(\roman*)]
\item The exact asymptotics of the codifference of the process on the
boundaries of the regimes (Remark~\ref{rem:boundary}(i)), where logarithmic
corrections are expected; the rate in
Theorem~\ref{thm:codiff-asymp}(b2) when $a_s\le-\frac12$; and the case
$s<0$.
\item The sign $\Psi<0$ for $n\ge2$, which is supported by the numerical
evidence of Section~\ref{sec:increments}, and, at the boundary $H_0=H_*$
with $n\ge2$, the expected equivalence
$r(m)\sim\al|\Ga(a_*)|^{-\al}m^{-1/(\al-1)}\log m$ (the case $n=1$ is
\cite[Theorem~3.1]{LevyTaqqu2012}).
\item The memory of the increments when $H$ converges to $H_+$ without
being eventually constant (Remark~\ref{rem:open-incr}).
\item Theorem~\ref{thm:LASS} under a Hölder-Zygmund condition of order larger
than $H(t_0)$ instead of $\mathcal C^n$ (Remark~\ref{rem:LASS-n1}).
\end{enumerate}
Broader directions include the local and pointwise Hölder regularity of
$n$-MFSM measured by $n$-th order increments (Remark~\ref{rem:holder}), the
extension to $\al\in(0,1]$, and the statistical estimation of $H(\cdot)$ and
$\al$ from discrete observations of $n$-th order increments, in the spirit
of \cite{AyacheHamonier2015,Dang2020}.

\end{document}